\documentclass[11pt]{article}
\usepackage[a4paper,margin=2.5cm]{geometry}
\usepackage{amsmath,amsthm,amssymb,amsfonts,amscd}
\usepackage{amsxtra}
\usepackage{eucal}
\usepackage{mathrsfs}
\usepackage{color}
\usepackage{xcolor}
\usepackage{hyperref,cleveref}
\usepackage{enumerate}
\usepackage{blkarray}
\usepackage{enumitem}
\usepackage{tcolorbox}
\usepackage{ifthen}
\usepackage{quiver} 

\usepackage{biblatex}
\usepackage{tikz-cd}

\theoremstyle{plain}
\newtheorem{thm}{Theorem}[section]
\newtheorem{lem}[thm]{Lemma}
\newtheorem{cor}[thm]{Corollary}
\newtheorem{prop}[thm]{Proposition}

\theoremstyle{definition}
\newtheorem{defn}[thm]{Definition}
\newtheorem{rmk}[thm]{Remark}
\newtheorem{ex}[thm]{Example}
\newtheorem{que}[thm]{Question}

\newcommand{\OG}{\mathrm{OG}}

\newcommand{\Gr}{\mathrm{Gr}}

\newcommand{\bC}{\mathbb{C}}
\newcommand{\bF}{\mathbb{F}}
\newcommand{\bK}{\mathbb{K}}

\newcommand{\bP}{\mathbb{P}}

\newcommand{\bS}{\mathbb{S}}
\newcommand{\bT}{\mathbb{T}}

\newcommand{\cA}{\mathcal{A}}
\newcommand{\cB}{\mathcal{B}}
\newcommand{\cC}{\mathcal{C}}
\newcommand{\cD}{\mathcal{D}}

\newcommand{\cT}{\mathcal{T}}

\newcommand{\ground}{[n]\cup [n]^*}

\newcommand{\supp}{\mathrm{supp}}
\newcommand{\ul}[1]{\underline{#1}}
\newcommand{\ol}[1]{\overline{#1}}
\newcommand{\symdiff}{\triangle}

\newcommand{\pf}{\mathrm{pf}}
\newcommand{\bilin}[2]{\langle #1, #2 \rangle}

\newcommand{\skewpair}[1]{\{#1, #1^*\}}
\newcommand{\skewpairs}[2]{\{#1, #1^*, #2, #2^*\}}
\newcommand{\eqcls}[1]{\left[ #1 \right]}

\newcommand{\WtoGP}{\Phi_{\mathrm{W}\to\mathrm{GP}}}
\newcommand{\GPtoO}{\Phi_{\mathrm{GP}\to\mathrm{O}}}
\newcommand{\OtoW}{\Phi_{\mathrm{O}\to\mathrm{W}}}
\newcommand{\WtoO}{\Phi_{\mathrm{W}\to\mathrm{O}}}

\newcommand{\indicator}[1]{\mathbf{1}(#1)}
\newcommand{\abs}[1]{\left| #1 \right|}
\newcommand{\smaller}[2]{| #1 < #2 |}

\newcommand{\smallereq}[2]{| #1 \le #2 |}

\newcommand{\hyper}{S}
\newcommand{\hypo}{S'}

\newcommand{\FC}{\mathrm{FC}}

\usepackage{authblk}
\usepackage{lineno}

\title{Grassmann--Pl\"ucker functions for orthogonal matroids}
\author[1]{Changxin Ding\thanks{\href{mailto:cding66@gatech.edu}{cding66@gatech.edu}. Changxin Ding was supported by the AMS-Simons Travel Grant.}}
\author[1]{Donggyu Kim\thanks{\href{mailto:donggyu@gatech.edu}{donggyu@gatech.edu}. Donggyu Kim was supported by the AMS-Simons Travel Grant.}}
\affil[1]{School of Mathematics, Georgia Institute of Technology, USA}

\begin{document}

\maketitle

\begin{abstract}
    We present a new cryptomorphic definition of orthogonal matroids with coefficients using Grassmann--Pl\"ucker functions.
    The equivalence is motivated by Cayley's identities expressing principal and almost-principal minors of a skew-symmetric matrix in terms of its Pfaffians. As a corollary of the new cryptomorphism, we deduce that each component of the orthogonal Grassmannian is parameterized by certain part of the Pl\"ucker coordinates.
\end{abstract}

\section{Introduction}

A matroid is a structure that abstracts the notion of linear independence in vector spaces. For matroids, there are several not obviously equivalent definitions, which are called \emph{cryptomorphisms}. For matroids with additional structures, e.g. valuated matroids and oriented matroids, their cryptomorphic characterizations do not follow directly from those of matroids. Recently, Baker and Bowler \cite{BB} introduced matroids over \emph{tracts} and studied the cryptomorphisms, which unify cryptomorphisms of (ordinary) matroids, valuated matroids, and oriented matroids. The key idea is to use the Pl\"{u}cker embedding of the Grassmannian. 

\emph{Orthogonal matroids} (also known as \emph{even $\Delta$-matroids}) generalize matroids. For the orthogonal Grassmannian $\OG(n,2n)$, there is a similar embedding that uses Pfaffians of skew-symmetric matrices, of which image is cut out by Wick relations. Jin and the second author \cite{JK} made use of this embedding to define orthogonal matroids over tracts and study the cryptomorphisms, which generalize the work in~\cite{BB}. 
As an application of~\cite{JK}, Baker and the authors studied the Jacobian groups of ribbon graphs~\cite{BDK2025}.

We introduce a new cryptomorphism for orthogonal matroids over tracts. 
The key idea is that the Pfaffians of a skew-symmetric matrix are determined by its principal and almost-principal minors due to Cayley's identities (cf. Remark~\ref{rmk: Cayley}).

\subsection{Main results}

Our first main result is as follows.
\begin{thm}\label{thm:main}
    For a tract $F$, there are natural bijections between: 
    \begin{enumerate}
        \item Orthogonal $F$-matroids, defined as the projective equivalence classes of Wick $F$-functions.
        \item Projective equivalence classes of restricted Grassmann--Pl\"{u}cker $F$-functions of type D.
        \item Orthogonal $F$-signatures of orthogonal matroids.
    \end{enumerate}
\end{thm}

The equivalence between the first and third was proven in \cite{JK}, and the second is the new cryptomorphism. Our result is even new in the realizable case. Indeed, when $F$ is the complex field, we have the following result. For the formal statement, see \S\ref{sec:rGP-field}.
For an $n$-element subset $I$ of $\ground:=\{1,2,\dots,n,1^*,2^*,\dots,n^*\}$, we call $I$ a \emph{transversal} if it contains exactly one of $i$ and $i^*$ for each $1\le i\le n$, and call $I$ an \emph{almost-transversal} if it contains exactly one pair of the form $\{i,i^*\}$. Also, recall that the orthogonal Grassmannian $\OG(n,2n)$ is the set of all $n$-dimensional isotropic subspaces of $\bC^{\ground}$ with respect to a symmetric
non-degenerate bilinear form. 

\begin{cor}\label{cor:rGP-embedding-intro}(Cor. \ref{cor:rGP-embedding})
    Let $p': \OG(n,2n) \to \bP^{2^n + n(n-1) 2^{n-2} - 1}(\bC)$ be the restricted Pl\"ucker embedding; that is, the image of $V\in\OG(n,2n)$ is the Pl\"ucker coordinates indexed by transversals and almost-transversals.    
    Then the image of each connected component of $\OG(n,2n)$ under $p'$ is cut out by the restricted Pl\"ucker relations along with certain linear relations.
\end{cor}

We also present analogous equivalences for weak orthogonal $F$-matroids.
As noted in \cite{BB}, there are (at least) two natural notions of matroids over tracts, called weak and strong $F$-matroids, and similar situations occur in the study of orthogonal matroids \cite{JK}.
The study of weak $F$-matroids has a long history, predating the modern terminology. When $K$ is a field, a weak $K$-matroid is exactly a vector that satisfies all the $3$-term Pl\"ucker relations and whose support forms a matroid. 
It is a folklore fact that such a vector satisfies all the Pl\"ucker relations; see~\cite{BJ2023} for a proof. In his paper \cite{Tutte1958}, Tutte studied chain-group representations of matroids and used this notion to find excluded-minors of binary and regular matroids. 
For any partial field $F$, the chain-group representations are (strong) $F$-circuit sets of matroids in the sense of \cite{BB}, which are cryptomorphic to (strong) $F$-matroids. A key step in Tutte's paper is to prove that weak $F$-circuit sets of matroids are in fact strong for any partial field $F$. 
More recently, weak $F$-matroids have appeared in the study of \emph{foundations} of matroids~\cite{BL2025b,BLZ2025} and orthogonal matroids~\cite{Jin2025}.

\begin{thm}\label{thm:main-weak}
    For a tract $F$, there are natural bijections between: 
    \begin{enumerate}
        \item Weak orthogonal $F$-matroids.
        \item Projective equivalence classes of weak restricted Grassmann--Pl\"{u}cker $F$-functions of type~D.
        \item Weak $F$-circuit sets of orthogonal matroids.
    \end{enumerate}
\end{thm}
Regarding the comparison of the two theorems, we remark that (strong) $F$-circuit sets are equivalent to (strong) orthogonal $F$-signatures of orthogonal matroids \cite{JK}. In the proof, we also suggest a new notion of weak orthogonal $F$-signatures, which is equivalent to weak $F$-circuit sets.

\subsection{Connections with other works.}

An analogous result of Corollary~\ref{cor:rGP-embedding-intro} for the Lagrangian Grassmannian, which is of type C, was proved by Boege et al.~\cite{BDKS}. 
They embedded the Lagrangian Grassmannian into the projective space through the restricted Pl\"ucker embedding described in Corollary~\ref{cor:rGP-embedding-intro}, and showed that the image of the embedding is cut out by certain quadratic equations (along with certain linear relations). Based on their work, the second author~\cite{Kim} simplified their list of quadratic relations to the restricted Pl\"ucker relations. Our results are inspired by \cite{BDKS} and \cite{Kim}, but for the orthogonal Grassmannian which is of type D. 

The readers who are familiar with \cite{Kim} might notice that orthogonal matroids are a subclass of the antisymmetric matroids defined in \cite{Kim}, and the cryptomorphisms for antisymmetric $F$-matroid for a tract $F$ have been studied in the same paper. However, orthogonal $F$-matroids are not a subclass of antisymmetric $F$-matroids because one mimics the orthogonal Grassmannian and the other mimics Lagrangian Grassmannian. For a precise discussion, see Remark~\ref{rmk: type D and type C}.

Our main contribution is to use the idea of restricted Pl\"ucker relations to study orthogonal $F$-matroids, which have previously been studied using Wick relations. We finally remark that it is difficult to find a type~C analogue of Wick relations, because there is no known notion of Pfaffians for symmetric matrices.

\subsection{Organization of the paper}

The paper is organized as follows.
In Section~\ref{sec:review}, we review the Baker--Bowler theory for matroids and orthogonal matroids over tracts.
In Section~\ref{sec:new-cryptomorphism}, we define the restricted Grassmann--Pl\"{u}cker $F$-functions of type D, study its support which we call \emph{even antisymmetric matroids}, and prove Theorem~\ref{thm:main} by presenting bijections between the three different objects. In Section~\ref{sec: weak}, we prove Theorem~\ref{thm:main-weak}. In Section~\ref{sec: mis}, we discuss two topics related to our theory: one concerning the connection to enveloping matroids, and the other concerning almost-principal minors and Pfaffian positivity.

\section{Matroids and orthogonal matroids over tracts}\label{sec:review}

We review the Baker--Bowler theory for matroids and orthogonal matroids with coefficients~\cite{BB,JK}. 
First of all, we recall the definition of tracts, which are field-like objects equipped with a relaxed additive structure that only captures whether sums are zero or not.

\begin{defn}[\cite{BB}]
    A \emph{tract} is a monoid $F = F^\times \sqcup \{0\}$, written multiplicatively, along with a set $N_F$ of formal sums of elements in $F$ modulo $0 + \sum a_i = \sum a_i$ satisfying the following axioms:
    \begin{enumerate}[itemsep=0ex, label=(T\arabic*)]
        \item $F^\times$ is an abelian group with identity $1$ and $0 \cdot a = a \cdot 0 = 0$ for all $a\in F$.
        \item $F\cap N_F = \{0\}$.
        \item\label{item:T3} There is a unique element $(-1) \in F^\times$ such that $1 + (-1) \in N_F$.
        \item\label{item:T4} For any $\alpha \in N_F$ and $c\in F^\times$, we have $c\alpha \in N_F$.
    \end{enumerate}
    The set $N_F$ is called the \emph{null set} of $F$. By~\ref{item:T3} and~\ref{item:T4}, one can deduce that $(-1)^2 = 1$ and if $a,b\in F$ with $a + b \in N_F$, then $b$ equals $-a := (-1) \cdot a$.
\end{defn}

Each field $K$ is naturally endowed with a tract structure by taking $K$ itself as a multiplicative monoid and $N_K$ as the set of formal sums $\sum a_i$ of elements that is zero as an element in the field $K$. For instance, $N_{\bF_2} = \{0, 1+1, 1+1+1+1, \dots \}$ is the set of sums with even number of $1$'s.

The \emph{Krasner hyperfield} is the tract $\bK = \{0,1\}$ with the usual multiplication, $0\cdot 0 = 0 \cdot 1 = 1 \cdot 0 = 0$ and $1\cdot 1 = 1$, and the null set $N_{\bK} = \{0, 1+1, 1+1+1, \dots\}$ consisting of $0$ and formal sums $1+\cdots+1$ with at least two $1$'s.

We refer the readers to~\cite{BB} for other basic examples, such as the sign hyperfield $\bS$ and the tropical hyperfield $\bT$.

\subsection{Matroids over tracts}
Let $E$ be a finite set with a linear ordering. We identify $E$ with $[n]:=\{1,2,\dots,n\}$.
Let $r$ be an integer with $0 \le r \le n$. For a subset $X\subseteq E$ and an element $i\in E$, we denote \[|X < i|:=\text{the number of elements in }X\text{ that are less than }i.\]

\begin{defn}\label{def:GP}
    A \emph{Grassmann--Pl\"{u}cker function} on $E$ of rank $r$ over a tract $F$ is a function $\varphi: \binom{E}{r} \to F$ satisfying the following conditions:
    \begin{enumerate}[label=\rm(GP\arabic*)]
        \item\label{item:GP1} $\varphi$ is not identically zero.
        \item\label{item:GP2} $\varphi$ satisfies the \emph{Grassmann--Pl\"{u}cker relations}; that is, for any $\hyper, \hypo \subseteq E$ with $|\hyper|=r+1$ and $|\hypo|=r-1$, we have 
        \[
            \sum_{i\in \hyper \setminus \hypo} (-1)^{|\hyper \symdiff \hypo < i|} \varphi(\hyper \setminus \{i\}) \varphi(\hypo \cup \{i\}) \in N_F.
        \]
    \end{enumerate}
\end{defn}

\begin{defn}
An \emph{$F$-matroid} is the (projective) equivalence class of a Grassmann-Pl\"ucker $F$-function, where two Grassmann-Pl\"ucker $F$-functions $\varphi_1$ and $\varphi_2$ are (projectively) equivalent if $\varphi_1=c\varphi_2$ for some $c\in F^\times$. 
\end{defn}

When $F = K$ is a field, $K$-matroids are exactly Pl\"ucker coordinates in $\bP^{\binom{n}{r}-1}(K)$. 
Also, $\bK$-matroids are ordinary matroids, as \ref{item:GP1} and~\ref{item:GP2} can be identified with the conditions that $\cB := \{B\in \binom{E}{r} \mid \varphi(B) \ne 0\}$ is nonempty and satisfies the (strong) basis exchange property.
By taking $F = \bT$ or $\bS$, we obtain valuated matroids or oriented matroids, respectively. 

\begin{defn}\label{def:sign}
    An \emph{$F$-signature} is a set $\cC$ of $F$-vectors $X\in F^E \setminus \{0\}$ such that $X \in \cC$ and $c\in F^\times$ imply $cX \in \cC$.
    We say $\cC$ is an \emph{$F$-signature of a matroid $M$} if $\supp(\cC) := \{\supp(X) \mid X\in \cC\}$ is the set of circuits of $M$. 
\end{defn}

\begin{defn}\label{def:dual-pair}
    Let $\cC$ and $\cD$ be $F$-signatures of a matroid $M$ and its dual $M^\perp$, respectively. We say $(\cC,\cD)$ is a \emph{dual pair} of $F$-signatures of $M$ if 
    \begin{enumerate}[label=\rm(DP)]
        \item\label{item:DP} $X \cdot Y := \sum_{i\in E} X(i) Y(i) \in N_F$ for all $X\in \cC$ and $Y\in \cD$.
    \end{enumerate}
\end{defn}

Given a linear subspace $V$ of $K^n$, let $\cC$ be the set of nonzero vectors in $V$ with minimal supports and let $\cD$ be the set of nonzero vectors in $V^\perp$ with minimal supports. Then $(\cC,\cD)$ is a dual pair of $K$-signatures of the matroid associated with $V$.

\begin{thm}[\cite{BB}]\label{thm:BB}
    There is a natural bijection between $F$-matroids and dual pairs of $F$-signatures of matroids.
\end{thm}

\subsection{Orthogonal matroids over tracts}\label{sec:review-orthogonal-matroids}

Jin and the second author~\cite{JK} extended the Baker--Bowler theory to orthogonal matroids by using the Wick embedding of the orthogonal Grassmannian $\OG(n,2n)$.
We first review the definition of orthogonal matroids.

\begin{defn}
    For a positive integer $n$, we denote $[n]^*:=\{1^*,2^*,\cdots,n^*\}$.
    \begin{enumerate}
        \item For $i\in [n]$, we define $(i^*)^* := i$. For $S \subseteq \ground$, we write $S^* := \{i^* \mid i\in S\}$.
        \item A subset $T$ of $\ground$ is called a \emph{transversal} if $|T\cap\{i,i^*\}|=1$ for every $i\in[n]$. 
        The set of all transversals of $\ground$ is denoted by $\cT_n$.
        \item A \emph{subtransversal} is a subset of a transversal.
    \end{enumerate}
\end{defn}

\begin{defn}\label{def:orthogonal-matroid}
    An \emph{orthogonal matroid} is a pair $( \ground,\cB)$ where $\cB \subseteq \cT_n$ satisfies the \emph{\textup{(}strong\textup{)} exchange axiom}: for any $B_1,B_2 \in \cB$ and $\skewpair{i} \subseteq B_1 \symdiff B_2$, there exists $\skewpair{j} \subseteq (B_1 \symdiff B_2) \setminus \skewpair{i}$ such that $B_1 \symdiff \skewpairs{i}{j}$ and $B_2 \symdiff \skewpairs{i}{j}$ are in $\cB$.
 
    We call $B\in \cB$ a \emph{basis} of the orthogonal matroid. 
    A \emph{circuit} is a minimal subtransversal that is not contained in any basis.
\end{defn}

By definition, for an orthogonal matroid $M=(\ground,\cB)$, the subsets $B\cap[n]$ with $B\in \cB$ have the same parity. If $M = (\ground,\cB)$ is an orthogonal matroid such that $|B\cap[n]|$ has the same size $r$ for all $B\in \cB$, then $M' = ([n],\{B\cap[n] \mid B\in \cB\})$ is a matroid of rank $r$. Conversely, one can reconstruct $M$ from the matroid $M'$. We call $M$ the \emph{lift} of $M'$.
Through this correspondence, the class of matroids is regarded as a subclass of orthogonal matroids. 
Important subclasses of orthogonal matroids include representable orthogonal matroids and ribbon-graphic orthogonal matroids.

\begin{ex}[\cite{Bouchet1988}]
    Let $\mathbf{A}$ be an $n$-by-$n$ skew-symmetric matrix over a field $K$.
    Then $\cB = \{([n]\setminus X) \cup X^* \mid \mathbf{A}_X \text{ is nonsingular}\}$ is the basis set of an orthogonal matroid, denoted $M(\mathbf{A})$, on $\ground$, where $\mathbf{A}_X$ is the principal submatrix of $\mathbf{A}$ with rows and columns indexed by $X \subseteq [n]$.
    Given an orthogonal matroid $M=(\ground,\cB)$ and a set $R\subseteq [n]$, a pair $M\symdiff R := (\ground, \cB')$ is an orthogonal matroid, where $\cB' = \{B\symdiff (R\cup R^*) \mid B\in \cB\}$.
    An orthogonal matroid $M$ is \emph{representable} over $K$ if it is equal to $M(\mathbf{A})\symdiff R$ for some matrix $\mathbf{A}$ over $K$ and some set $R$.
    Bouchet~\cite{Bouchet1988} showed that a matroid is representable over $K$ if and only if its lift is representable over~$K$.
\end{ex}

\begin{ex}
For any graph, one can associate a matroid to it, known as a graphic matroid. Analogously, one can associate an orthogonal matroid to any \emph{ribbon graph} (i.e., a graph $G$ with a cellular embedding into an orientable surface $\Sigma$); see \cite{Bouchet1987b, BBGS2000}. 
\end{ex}

We recall the definition of the orthogonal Grassmannian and Wick embedding. 
\begin{defn}
    Let $\bC^{2n}(=\bC^{\ground})$ be the $2n$-dimensional space equipped with the symmetric bilinear form \[\bilin{X}{Y} = \sum_{i=1}^n  \left(X(i) Y(i^*) + X(i^*) Y(i) \right).\] A subspace $V$ of $\bC^{2n}$ is \emph{isotropic} if $\bilin{X}{Y} = 0$ for all $X,Y\in V$.
    The \emph{orthogonal Grassmannian} $\OG(n,2n)$ is the set of $n$-dimensional isotropic subspaces in $\bC^{2n}$.
\end{defn}

The space $\OG(n,2n)$ is embedded into $\bP^{2^n - 1}(\bC)$ via the Wick embedding $w$. For simplicity, we assume that $V$ is the row space of an $n$-by-$2n$ matrix $(\mathbf{I}_n \mid \mathbf{A})$ with columns indexed by $1,2,\dots,n,1^*,2^*,\dots,n^*$ in this order. Then $\mathbf{A}$ is skew-symmetric. The \emph{Wick coordinates} $w(V) \in \bP^{2^n - 1}(\bC)$ are defined as 
\[
    w(V)(T) = \pf(\mathbf{A}_{T\setminus[n]}) \text{ for } T\in \cT_n,
\]
where $\mathbf{A}_{T\setminus[n]}$ is the principal submatrix of $\mathbf{A}$ with columns indexed by $T\setminus[n]$.

Analogous to the Grassmannian and the Pl\"ucker relations, the image of $\OG(n,2n)$ under the \emph{Wick embedding} is cut out by the \emph{Wick relations}, which are used to define orthogonal matroids over tracts in~\cite{JK}.

\begin{defn}
    A \emph{\textup{(}strong\textup{)} Wick function} on $\ground$ over a tract $F$ is a function $\psi : \cT_n \to F$ satisfying the following conditions:
    \begin{enumerate}[label=\rm(W\arabic*)]
        \item\label{item:W1} $\psi$ is not identically zero.
        \item\label{item:W2} $\psi$ satisfies the \emph{Wick relations}; that is, for any $T, T'\in \cT_n$, we have 
        \[
            \sum_{i\in(T \symdiff T')\cap [n]} (-1)^{|(T \symdiff T')\cap [n]<i|} \psi(T \symdiff\skewpair{i}) \psi(T' \symdiff\skewpair{i})\in N_F.
        \]
    \end{enumerate}
\end{defn}

\begin{defn}
An \emph{orthogonal $F$-matroid} is an equivalence class of Wick $F$-functions, where two Wick functions are equivalent if one is a nonzero multiple of the other.    
\end{defn}

\begin{lem}[\cite{JK}]
    For an orthogonal $F$-matroid $[\psi]$ on $\ground$, the pair $(\ground, \{B\in \cT_n \mid \psi(B)\ne 0\}$) is an orthogonal matroid.
\end{lem}

Orthogonal $\bC$-matroids are the same as Wick coordinates in $\bP^{2^n - 1}(\bC)$, and orthogonal $\bK$-matroids are the same as ordinary orthogonal matroids.
By~{\cite[Prop.~3.4]{JK}}, orthogonal $F$-matroids whose underlying orthogonal matroids are lifts of matroids (see the paragraph below Definition~\ref{def:orthogonal-matroid}) are the same as $F$-matroids.

\begin{defn}
Let $\cC \subseteq F^{\ground}$ be an $F$-signature (see Definition~\ref{def:sign}). If $\supp(\cC)$ is the set of circuits of an orthogonal matroid $M$ on $\ground$, then we call $\cC$ an \emph{$F$-signature of $M$}.   Furthermore, an $F$-signature $\cC$ of an orthogonal matroid is called \emph{orthogonal} if
    \begin{enumerate}[label=\rm(O)]
        \item\label{item:O} $\langle X, Y \rangle := \sum_{i\in\ground} X(i) Y(i^*) \in N_F$ for all $X,Y\in \cC$.
    \end{enumerate}
\end{defn}

\begin{ex}
If $V$ is an $n$-dimensional isotropic subspace of $\bC^{2n}$, then the set of nonzero vectors in $V$ whose supports are minimal subtransversals forms an orthogonal $\bC$-signature of the orthogonal matroid associated with $V$. It coincides with the chain-group representation (over $\bC$) of even $\Delta$-matroids in~\cite{Bouchet1988}.
\end{ex}

If two vectors in an orthogonal $F$-signature have the same support, then they differ by a scalar. Actually, we have the following stronger result. 

\begin{lem}\label{lem:scalar}
Let $\cC$ be an orthogonal $F$-signature of an orthogonal matroid. 
If $X_1$ and $X_2$ in~$\cC$ have the same support, then $X_1=c X_2$ for some $c\in F^\times$.  
\end{lem}
\begin{proof}
    The same result with a weaker assumption than \ref{item:O} is proven in~\cite[Lem.~4.13]{JK}.
\end{proof}

The notion of orthogonal $F$-signatures of orthogonal matroids generalizes dual pairs of $F$-signatures of matroids.

A bijection between orthogonal $F$-matroids and orthogonal $F$-signatures of orthogonal matroids is constructed in~\cite{JK}. We briefly review this construction with a slightly different convention.
Let $\psi$ be a Wick $F$-function. For each $T\in \cT_n$ such that $T\symdiff \skewpair{x}$ is a basis for some $x\in [n]$, let $X_T \in F^{\ground}$ be the vector such that \[X_T(e)  = (-1)^{\smallereq{T^*\cap[n]}{\ol{e}}} \psi(T\symdiff \skewpair{e})\footnote{In~\cite{JK}, $X_T$ is defined by ratios $X_T(e)/X_T(x)$, which is equivalent to, up to rescaling, defining it as $X_T(e)  = (-1)^{\smallereq{T\cap[n]}{\ol{e}}} \psi(T\symdiff \skewpair{e})$.}\] if $e\in T$ and $X_T(e) = 0$ otherwise, where $\ol{e}$ is the unique element in $\{e,e^*\} \cap[n]$. 
Let $\cC$ be the set of all nonzero scalar multiples of such vectors, i.e., \[\cC = \{c X_T \mid c\in F^\times, \ T\in \cT_n, \ T\symdiff \skewpair{x} \in \supp(\psi) \text{ for some } x\in [n]\},\] and we define \[\WtoO(\eqcls{\psi}) := \cC.\]

\begin{thm}[\cite{JK}]\label{thm:JK}
    $\WtoO$ is a bijection between orthogonal $F$-matroids and orthogonal $F$-signatures.
\end{thm}

\section{New cryptomorphism for orthogonal matroids over tracts}\label{sec:new-cryptomorphism}

In \S\ref{sec:new}, we define restricted Grassmann--Pl\"{u}cker functions over a tract, which will be our new equivalent definition of orthogonal $F$-matroids in the main theorem. 
In \S\ref{sec:rGP-field}, we explain how restricted Grassmann--Pl\"ucker functions over $\bC$ arise from the Pl\"ucker embedding of the orthogonal Grassmannian. In \S\ref{sec:rGP-support}, we study the support of a restricted Grassmann--Pl\"ucker function, which we call an even antisymmetric matroid. 
In \S\ref{sec:proof-main-theorem}, we prove Theorem~\ref{thm:main}

\subsection{The new  cryptomorphism}\label{sec:new}
\begin{defn}
    \begin{enumerate}
        \item For a logical statement $S$, let 
        
        \[
\indicator{S} = \begin{cases}
    1, & \text{if }S \text{ is true}, \\
    0, & \text{if }S \text{ is false}.
\end{cases}
\]

        \item For any $i\in\ground$, the set $\skewpair{i}$ is called a \emph{skew pair}. 
        \item An $n$-subset $A$ of $\ground$ is an \emph{almost-transversal} if it contains exactly one skew pair. We denote by $\cA_n$ the collection of all almost-transversals.
        \item For $\sigma \in \{0,1\}$, let $\cT_n^\sigma$ be the collection of transversals $T$ such that $\abs{T\cap [n]^*} \equiv \sigma \pmod{2}$.
        \item A subset $S$ of $\ground$ is called a \emph{hyper-transversal} if it is obtained from a transversal by adding one element; it is called a \emph{hypo-transversal} if it is obtained by removing one element from a transversal.
    \end{enumerate}
\end{defn}

Recall that we are using the order
\[1 < 1^* < 2 < 2^* < \dots < n < n^*,\]
so that $|X < i|$ denotes the number of elements of $X$ smaller than $i$ with respect to this order, for a subset $X$ and an element $i$ of $\ground$.

\begin{defn}\label{defn:rGP}
    A \emph{restricted Grassmann--Pl\"{u}cker function \textup{(}of type D\textup{)}} over a tract $F$ is a function $\varphi: \cT_n \cup \cA_n \to F$ satisfying the following conditions:
    \begin{enumerate}[label=\rm(rGP\arabic*)]
        \item\label{item:rGP1} $\varphi$ is not identically zero.
        \item\label{item:rGP2} $\varphi$ satisfies the \emph{restricted Grassmann--Pl\"{u}cker relations}; that is, for any hyper-transversal~$\hyper$ and hypo-transversal~$\hypo$, we have 
        \[
            \sum_{i\in \hyper \setminus \hypo} 
            (-1)^{\abs{\hyper \symdiff \hypo < i}} 
            \varphi(\hyper \setminus \{i\}) 
            \varphi(\hypo \cup \{i\}) \in N_F.
        \]
        
        \item\label{item:rGP3} There is (a unique) $\sigma \in \{0,1\}$ such that $\varphi(B) = 0$ for any $B\in \cT_n^{1-\sigma}$.
        \item\label{item:rGP4} For any $B\in \cT_n^\sigma$ and distinct $i,j\in [n]$, we have
        \[
            \varphi(B\setminus\skewpair{i} \cup\skewpair{j})=(-1)^{ \indicator{i\in B} + \indicator{j\in B}}            \varphi(B\cup\skewpair{i} \setminus\skewpair{j}).
        \]
    \end{enumerate}
    We denote the equivalence class of $\varphi$ by \[\eqcls{\varphi}:=\{c \varphi\mid c\in F^\times\}.\] 
\end{defn}

\begin{rmk}
It is well known that the orthogonal Grassmannian $\OG(n,2n)$ has two connected components, each of which is a \emph{spinor variety}; see \cite{Rincon2012}. The two choices of $\sigma$ in \ref{item:rGP3} correspond to these two components.    
\end{rmk}

The relation in \ref{item:rGP4} says that the almost-transversals $B\setminus\skewpair{i} \cup\skewpair{j}$ and $B\cup\skewpair{i} \setminus\skewpair{j}$ may only differ by a sign. Here is an alternative way to describe it. 

\begin{lem}\label{lem: rgp4}
Suppose $\varphi$ satisfies \ref{item:rGP4}.
Let $A\in\cA_n$ with $\skewpair{i}\cap A=\emptyset$ and $\skewpair{j}\subseteq A$. Then 
\[\varphi(A)=(-1)^{|A\cap [n]^*|+1-\sigma}\varphi(A\cup\skewpair{i} \setminus\skewpair{j}).\]      
\end{lem}
\begin{proof}
Without loss of generality, we may assume that $i,j\in [n]$. Let $B$ be the transversal $A\setminus\{j\}\cup\{i\}$ or $A\setminus\{j^*\}\cup\{i\}$ such that $B\in\cT_n^\sigma$. By \ref{item:rGP4}, 
\begin{align*}
\varphi(A) & =\varphi(B\setminus\skewpair{i} \cup\skewpair{j})=(-1)^{ \indicator{i\in B} + \indicator{j\in B}}            \varphi(B\cup\skewpair{i} \setminus\skewpair{j})\\
& =(-1)^{ \indicator{i\in B} + \indicator{j\in B}}\varphi(A\cup\skewpair{i} \setminus\skewpair{j}).    
\end{align*}
For the sign, we have \[|B\cap[n]^*|-|A\cap[n]^*|=\indicator{i^*\in B} + \indicator{j^*\in B}-1=1-\indicator{i\in B} + \indicator{j\in B},\]
and hence
\[\indicator{i\in B} + \indicator{j\in B}\equiv |A\cap[n]^*|+1-\sigma \pmod{2}.
\qedhere \]
\end{proof}

\begin{rmk}
Restricted Grassmann--Pl\"{u}cker functions also satisfy some other Grassmann--Pl\"{u}cker relations except \ref{item:rGP2}; see Proposition~\ref{prop:GP on TA}.    
\end{rmk}

\subsection{Restricted Grassmann--Pl\"ucker functions over the complex field}\label{sec:rGP-field}

Via the Pl\"ucker embedding, the orthogonal Grassmannian $\OG(n,2n)$ is a subvariety of the Grassmannian $\Gr(n,2n)$. It has the two connected components, say $\OG^+(n,2n)$ and $\OG^-(n,2n)$, where $\OG^+(n,2n)$ is the component containing $V\in \OG(n,2n)$ with $p(V)([n]) \ne 0$, and $\OG^-(n,2n)$ is the other component.
Each of the components is cut out by the Pl\"ucker relations along with certain linear relations arising from the isotropic condition.
Let $p(V) \in \bP^{\binom{2n}{n}-1}(\bC)$ be the Pl\"ucker coordinates associated with $V$ defined by
\[
    p(V)(I) := \det(\mathbf{A}_I) 
    \text{ for $n$-subsets } I \text{ of } \ground,
\]
where $\mathbf{A}$ is an $n$-by-$2n$ matrix whose row space is $V$ and whose columns are indexed by $1 < 1^* < 2 < 2^* < \dots < n < n^*$ in order, and $\mathbf{A}_I$ is the submatrix of $\mathbf{A}$ with columns indexed by $I$.
In the case that $\mathbf{A}_{[n]}$ is the identity matrix (so, $V\in OG^+(n,2n)$), the submatrix $\mathbf{A}_{[n]^*}$ is skew-symmetric and thus for all $B\in \cT_n^0$ and distinct $i,j\in[n]$,
\begin{equation*}
    p(V)(B\setminus\skewpair{i} \cup\skewpair{j})
    =
    (-1)^{m}
    p(V)(B\cup\skewpair{i} \setminus\skewpair{j})
    \tag{$\ast$}\label{eq:linear-relation}
\end{equation*}
where $m = \indicator{i\in B} + \indicator{j\in B}$, 
and $p(V)(B) = 0$ for any $B\in \cT_n^{1}$.
One can deduce the same linear equations for any $V\in OG^+(n,2n)$, and one can also deduce the analogous linear equations for $V\in OG^-(n,2n)$ by interchanging $\cT_n^0$ and $\cT_n^1$.

\begin{ex}
    Let $\mathbf{A}$ and $\mathbf{A}'$ be the following $2$-by-$4$ matrices:
    \[
        \mathbf{A} = \begin{pmatrix}
            1 & 0 & 0 & a \\
            0 &-a & 1 & 0 \\
        \end{pmatrix}
        \ \text{ and } \
        \mathbf{A}' = \begin{pmatrix}
            0 & 1 & 0 & a \\
           -a & 0 & 1 & 0 \\
        \end{pmatrix},
    \]
    where $a\in \bC$ and the columns are indexed by $1 < 1^* < 2 < 2^*$ in order.
    We denote by $V$ and $V'$ the row spaces of $\mathbf{A}$ and $\mathbf{A}'$, respectively. Then $V\in \OG^+(2,4)$ and $V'\in \OG^-(2,4)$.
    If we take $B = \{1,2\} \in \cT_n^0$, $i=1$, and $j=2$, then we have:
    \begin{align*}
        p(V)(B\setminus\skewpair{1} \cup\skewpair{2}) &= \det(\mathbf{A}_{\skewpair{2}}) = -a,
        \\
        p(V)(B\cup\skewpair{1} \setminus\skewpair{2}) &= \det(\mathbf{A}_{\skewpair{1}}) = -a.
    \end{align*}
    Thus, the linear relation~\eqref{eq:linear-relation} holds for this choice of $B$, $i$, and $j$.
    We can also verify \eqref{eq:linear-relation} for $V'$ applied to $B' = \{1^*,2\} \in \cT_n^1$, $i=1$, and $j=2$, as we have the following:
    \begin{align*}
        p(V')(B'\setminus\skewpair{1} \cup\skewpair{2}) = \det(\mathbf{A}'_{\skewpair{2}}) &= -a,
        \\
        p(V')(B'\cup\skewpair{1} \setminus\skewpair{2}) = \det(\mathbf{A}'_{\skewpair{1}}) &= a.
    \end{align*}
\end{ex}

Theorem~\ref{thm:main} applied to $F=\bC$ implies that $V$ can be recovered even if we restrict $p(V)$ to the coordinates indexed by $\cT_n \cup \cA_n$.

\begin{cor}\label{cor:rGP-embedding}(Cor. \ref{cor:rGP-embedding-intro})
    Let $p' : \OG(n,2n) \to \bP^{2^n + n(n-1) 2^{n-2} - 1}(\bC)$ be the restricted Pl\"ucker embedding defined by
    \[p'(V)=(p(V)(I)\colon I\in \cT_n \cup \cA_n).\]
    Then the image of $\OG^+(n,2n)$ under $p'$ is cut out by the restricted Pl\"ucker relations along with the linear relations:
    for each $B\in \cT_n^0$ and distinct $i,j\in[n]$,
    \[
        X_{B\setminus\skewpair{i} \cup\skewpair{j}} =
        (-1)^{\indicator{i\in B} + \indicator{j\in B}}
        X_{B\cup\skewpair{i} \setminus\skewpair{j}},
    \]
    and $X_B= 0$ for any $B\in \cT_n^{1}$.
    The same holds for $\OG^-(n,2n)$ with $\cT_n^0$ and $\cT_n^1$ interchanged.
\end{cor}

\begin{rmk}
    When we take the Pl\"ucker embedding with a different ordering of columns, the linear relation~\eqref{eq:linear-relation} has different signs.
    For instance, the ordering $1 < 2 < \dots < n < 1^* < 2^* < \dots < n^*$ induces 
    \begin{equation*}
        p(V)(B\setminus\skewpair{i} \cup\skewpair{j})
        =
        (-1)^{1+i+j + \indicator{i\in B} + \indicator{j\in B}}
        p(V)(B\cup\skewpair{i} \setminus\skewpair{j})
    \end{equation*}
    for $V\in \OG^+(n,2n)$, $B\in \cT_n^0$, and distinct $i,j\in[n]$.
    The same holds for $V\in \OG^-(n,2n)$ and $B\in \cT_n^1$. 
\end{rmk}

\begin{rmk}
    In~\cite{Kim}, the second author introduced a \emph{restricted Grassmann--Pl\"ucker function} over $F$, which shares the same name as Definition~\ref{defn:rGP} but differs in the axioms \ref{item:rGP3} and \ref{item:rGP4}. More precisely, it is defined as a function $\varphi: \cT_n \cup \cA_n \to F$ satisfying \ref{item:rGP1} and \ref{item:rGP2} along with the following condition in place of \ref{item:rGP3} and \ref{item:rGP4}:
    \begin{enumerate}[label=\rm(rGP3$'$)]
        \item\label{item:rGP3'} For all $B \in \cT_n$ and distinct $i,j\in[n]$,
        \[
            \varphi(B\setminus\skewpair{i}\cup\skewpair{j}) = (-1)^{i+j} \varphi(B\cup\skewpair{i}\setminus\skewpair{j}).
        \]
    \end{enumerate}
    This notion generalizes the embedding of the Lagrangian Grassmannian $\text{Lag}(n,2n)$ using principal and almost-principal minors of symmetric matrices~\cite{BDKS}, under the ordering convention $1 < 2 < \dots < n < 1^* < 2^* < \dots < n^*$. 
    Although we will not use the notion in~\cite{Kim} in this paper, we refer to it as a restricted Grassmann--Pl\"ucker function of \emph{type C} to distinguish it from Definition~\ref{defn:rGP}.
\end{rmk}

\subsection{The support of a restricted Grassmann--Pl\"ucker function}\label{sec:rGP-support}
The support of a restricted Grassmann--Pl\"{u}cker function forms a subclass of antisymmetric matroids (defined and studied in \cite{Kim}). We call this subclass \emph{even antisymmetric matroids}, which are in bijection with orthogonal matroids in a natural way.

\begin{defn}[\cite{Kim}]\label{def: AM}
    An \emph{antisymmetric matroid} is a pair $M = (\ground, \cB)$ such that $\emptyset \ne \cB \subseteq \cT_n\cap \cA_n$ and the following hold:
    \begin{enumerate}
        \item For any hyper-transversal $\hyper$ and hypo-transversal $\hypo$, there are no or at least two $i \in \hyper \setminus \hypo$ such that $\hyper \setminus\{i\}$ and $\hypo \cup \{i\}$ are in $\cB$.
        \item For any transversal $T$ and distinct $i,j\in[n]$, we have $T\setminus\skewpair{i}\cup\skewpair{j} \in \cB$ if and only if $T\cup\skewpair{i}\setminus\skewpair{j} \in \cB$.
    \end{enumerate}
    Elements in $\cB$ are called \emph{bases} of $M$. A \emph{circuit} of $M$ is a minimal subset $C$ of $\ground$ such that $C$ contains at most one skew pair and $C$ is not a subset of any basis of $M$. 
\end{defn}

\begin{defn}
An antisymmetric matroid is \emph{even} if $\cB\cap \cT_n \subseteq \cT_n^\sigma$ for some $\sigma\in\{0,1\}$. 
\end{defn}

\begin{lem}
    The support of a restricted Grassmann--Pl\"ucker function over a tract forms an even antisymmetric matroid.
    Moreover, (equivalence classes of) restricted Grassmann--Pl\"ucker functions over $\bK$ are in one-to-one correspondence with even antisymmetric matroids.
\end{lem}
\begin{proof}
By \ref{item:rGP1}, \ref{item:rGP2}, and \ref{item:rGP4}, the support of a restricted Grassmann--Pl\"ucker function forms an antisymmetric matroid. Here we can replace ``any transversal $T$'' in Def.~\ref{def: AM} with ``any $T\in \cT_n^\sigma$'' for any choice of $\sigma\in\{0,1\}$ because the definition will still cover all pairs of almost-transversals $T\setminus\skewpair{i}\cup\skewpair{j}$ and $T\cup\skewpair{i}\setminus\skewpair{j}$. Then by \ref{item:rGP3}, the antisymmetric matroid is even.

The second assertion also follows directly from the definition. 
\end{proof}

Even antisymmetric matroids are in one-to-one correspondence with orthogonal matroids, and they have identical circuit sets. This result is implicitly proved in \cite{Kim}. 
\begin{thm}[\cite{Kim}]\label{thm:support}
    Let $\cB \subseteq \cT_n^\sigma$ for some $\sigma \in \{0,1\}$.
    Let $\cB'$ be the set of almost-transversals $A$ such that either
    \begin{enumerate}
        \item $A\setminus \{i\} \cup \{j\}$ and $A\setminus \{i^*\} \cup \{j^*\}$ are in $\cB$, or 
        \item $A\setminus \{i^*\} \cup \{j\}$ and $A\setminus \{i\} \cup \{j^*\}$ are in $\cB$,
    \end{enumerate}
    where $i,j\in[n]$ are the elements such that $\skewpair{i} \subseteq A$ and $\skewpair{j} \cap A = \emptyset$.
    Then the map \[M=(\ground, \cB)\mapsto M'=(\ground, \cB\cup \cB')\] is bijection between orthogonal matroids and even antisymmetric matroids on $\ground$. Moreover, the orthogonal matroid $M$ and even antisymmetric matroid $M'$ have the same circuit set. 
\end{thm}
\begin{proof}
For any orthogonal matroid $M=(\ground, \cB)$, \cite[Theorem 4.13]{Kim} (together with its proof) says that $(\ground, \cB\cup \cB')$ is an antisymmetric matroid and there is no other subset $\cB''$ of almost-transversals such that $(\ground, \cB\cup \cB'')$ is also an antisymmetric matroid. Because $M$ is an orthogonal matroid, $M'$ must be even. 

The map is clearly injective. To show the surjectivity, we start with an even antisymmetric matroid $M'=(\ground, \cB\cup \cB')$, where $\cB \subseteq \cT_n^\sigma$ for some $\sigma \in \{0,1\}$ and $\cB' \subseteq \cA_n$. By \cite[Proposition 4.13]{Kim}, $M=(\ground, \cB)$ is a symmetric matroid, i.e., for any $B_1,B_2\in\cB$ and any $x\in B_1\setminus B_2$, there exists $y\in B_1\setminus B_2$ such that $(B_1\symdiff\skewpair{x})\symdiff\skewpair{y}\in\cB$. Since $\cB \subseteq \cT_n^\sigma$, we have $\skewpair{x}\neq\skewpair{y}$ and hence $M$ is an orthogonal matroid. Then the uniqueness part of \cite[Theorem 4.13]{Kim} guarantees that $M\mapsto M'$.  

It is shown in the proof of \cite[Theorem 4.13]{Kim} that $M$ and $M'$ have the same circuit set. 
\end{proof}

\begin{rmk}\label{rmk: type D and type C}
Although orthogonal matroids can be identified with even antisymmetric matroid, the orthogonal $F$-matroids $M$ defined in \cite{JK} are in general different from the even antisymmetric $F$-matroids in \cite{Kim}. More precisely, we claim that their $F$-circuit sets are the same only when $1=-1$ in the tract $F$ or the underlying orthogonal matroid of $M$ is the lift of a matroid. This is why we need to replace \ref{item:rGP3'} in \cite{Kim} with \ref{item:rGP3} and \ref{item:rGP4}. As a consequence, there are two different notions of even antisymmetric matroids over tracts, one is for type C, and the other is for type D. 
\end{rmk}

We will need the following technical lemmas later. The first one follows from {\cite[Lemma~3.2]{Kim}}.
\begin{lem}\label{lem:hypo/hyper-transversals}
    Let $M$ be an antisymmetric matroid. The following holds:
    \begin{enumerate}
        \item For any hyper-transversal $\hyper$ and the skew pair $\skewpair{i}$ with $\skewpair{i} \subseteq \hyper$, if $\hyper \setminus\{i'\}$ is a basis for some $i'\in \hyper$, then $\hyper \setminus\{i\}$ or $\hyper \setminus\{i^*\}$ is a basis.
        \item For any hypo-transversal $\hypo$ and the skew pair $\skewpair{j}$ with $\skewpair{j}\cap \hypo = \emptyset$, if $\hypo \cup\{j'\}$ is a basis for some $j'\notin \hypo$, then $\hypo \cup\{j\}$ or $\hypo \cup\{j^*\}$ is a basis.
    \end{enumerate}
\end{lem}

The next lemma is about the fundamental circuits.  

\begin{lem}\label{lem:two circuits}
Let $M=(\ground, \cB)$ and $M'=(\ground, \cB\cup \cB')$ be a pair of an orthogonal matroid and an even antisymmetric matroid as in Theorem~\ref{thm:support}. Let $T\in\cT_n$ and $i\in T$ be such that $T\symdiff\skewpair{i}\in\cB$, and let $S=T\cup\{i^*\}$.

\begin{enumerate}
    \item\cite{BMP2003} For $M$, there exists a unique circuit $C\subseteq T$ and \[C=\{j\in T: T\symdiff\skewpair{j}\in\cB\}.\]
    \item\cite[Lemma 3.7]{Kim} For $M'$, there is a unique circuit $C'\subseteq S$ and \[C'=\{j\in S: S\setminus\{j\}\in\cB\}.\]
    \item $C=C'$.
\end{enumerate}
\end{lem}

\begin{proof}
By Theorem~\ref{thm:support}, $C$ is also a circuit of $M'$. Since $C\subseteq S$, we have $C=C'$.
\end{proof}

\begin{defn}\label{def: fundamental circuit}
The circuit $C$ in Lemma~\ref{lem:two circuits} is called the \emph{fundamental circuit} of the orthogonal matroid $M$ with respect to the basis $T\symdiff\skewpair{i}$ and $i$. The circuit $C'$ is called the \emph{fundamental circuit} of the even antisymmetric matroid $M'$ with respect to the basis $T\symdiff\skewpair{i}$ and $i$. 
\end{defn} 
The fundamental circuits of an antisymmetric matroid (not necessarily even) are defined in \cite{Kim}. Our definition is a special case. 
For matroids and orthogonal matroids, every circuit is a fundamental circuit. The counterpart for antisymmetric matroids also holds \cite[Lemma 3.6]{Kim}. Here, we only state the necessary part as a lemma. 

\begin{lem}\label{lem: all fundamental}
Every circuit of an even antisymmetric matroid is a fundamental circuit.    
\end{lem}

\subsection{Proof of Theorem~\ref{thm:main}}\label{sec:proof-main-theorem}

Theorem~\ref{thm:main} says that orthogonal $F$-matroids (defined as classes of Wick $F$-functions), classes of restricted Grassmann--Pl\"ucker $F$-functions, and orthogonal $F$-signatures are equivalent; see the diagram below.  
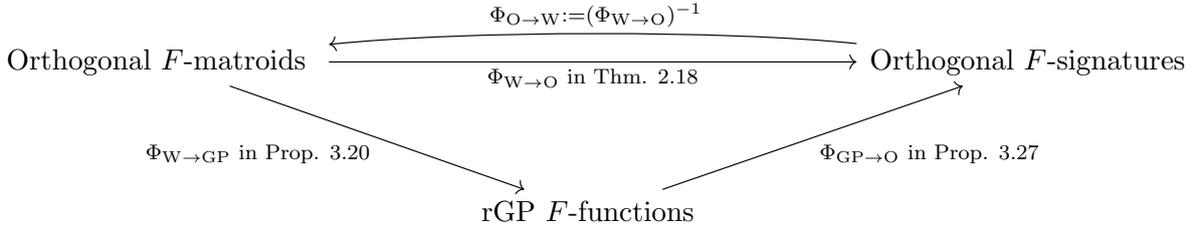
\[\begin{tikzcd}
	{\text{Orthogonal }F\text{-matroids }} && {} && {\text{Orthogonal }F\text{-signatures }} \\
	\\
	&& {\text{rGP }F\text{-functions }}
	\arrow["\WtoO\text{ in Thm. }\ref{thm:JK}"', from=1-1, to=1-5]
	\arrow["\WtoGP\text{ in Prop. }\ref{prop:WtoGP}"', from=1-1, to=3-3]
	\arrow["\OtoW:=(\WtoO)^{-1}"', curve={height=12pt}, from=1-5, to=1-1]
	\arrow["\GPtoO\text{ in Prop. }\ref{prop:GPtoO}"', from=3-3, to=1-5]
\end{tikzcd}\]
By \cite{JK}, we already have the bijection $\WtoO$. In \S\ref{sec:WtoGP} and \S\ref{sec:GPtoO}, we will define $\WtoGP$ and $\GPtoO$, respectively. In \S\ref{sec:2identitymaps}, we show that both $\GPtoO \circ \WtoGP \circ \OtoW$ and $\WtoGP \circ \OtoW \circ \GPtoO$ are identity maps, which implies Theorem~\ref{thm:main}.

We will often abbreviate Grassmann--Pl\"ucker as GP. Also recall that for $i\in \ground$, we write $\ol{i}$ for the unique element in $\{i,i^*\}\cap [n]$. 
We will be using the order
\[1 < 1^* < 2 < 2^* < \dots < n < n^*.\]

\subsubsection{Wick functions to restricted GP functions}\label{sec:WtoGP}

Let $\psi:\cT_n \to F$ be a Wick function. Recall that the support of $\psi$ is an orthogonal matroid, and hence there exists  $\sigma\in\{0,1\}$ such that all the bases belong to $\cT_n^\sigma$. We define a function $\varphi: \cT_n \cup \cA_n \to F$ as follows. 
Firstly, \[
    \varphi(B) = \psi(B)^2\text{ for each }B\in \cT_n.
\]
Then for each $A\in\cA_n$, 
\begin{equation}\label{eq:GPA}
    \varphi(A) = (-1)^{\smaller{B\cap[n]}{i} + \smaller{B\cap[n]}{j} + \indicator{i\in B} +1}\psi(B)\psi(B\symdiff\skewpairs{i}{j}),
\end{equation}
where $i,j\in [n]$ are the elements such that $\skewpair{i}\cap A=\emptyset$ and $\skewpair{j}\subseteq A$, and $B$ is one of the two transversals in $\cT_n^\sigma$ such that $B \setminus \skewpair{i} \cup \skewpair{j}=A$. The other transversal is $B'=B\symdiff\skewpairs{i}{j}$, and our formula \eqref{eq:GPA} does not depend on the choice of $B$. Indeed, we have
\begin{align*}
\smaller{B\cap[n]}{i}-\smaller{B'\cap[n]}{i}&\equiv \smaller{(B\symdiff B')\cap[n]}{i}\equiv\indicator{j<i},\\
\smaller{B\cap[n]}{j}-\smaller{B'\cap[n]}{j}&\equiv \smaller{(B\symdiff B')\cap[n]}{j}\equiv\indicator{i<j},\\
\indicator{i\in B}-\indicator{i\in B'}&\equiv 1,\\
0& \equiv \indicator{j<i}+\indicator{i<j}+1 \pmod{2}.
\end{align*}

It is easy to see that $\varphi$ satisfies \ref{item:rGP1} and \ref{item:rGP3}. 

\begin{rmk}
By our definition of $\varphi$, the support of $\psi$ corresponds to the support of $\varphi$ via the bijection in Theorem~\ref{thm:support}.  
\end{rmk}

We will show in Lemma~\ref{lem:rGP4} and Lemma~\ref{lem:rGP2} that $\varphi$ also satisfies \ref{item:rGP2} and \ref{item:rGP4}. Hence, we obtain the following map. 

\begin{prop}\label{prop:WtoGP}
    The map defined by\[\WtoGP (\eqcls{\psi}) := \eqcls{\varphi}\]sends an orthogonal $F$-matroid $\eqcls{\psi}$ to an equivalence class of restricted Grassmann--Pl\"{u}cker $F$-functions.
\end{prop}

\begin{rmk}\label{rmk: Cayley}
    When $F$ is a field, the map $\WtoGP$ is derived from Cayley's identities \cite[Equation~(6.3)]{Knuth}: for an $n\times n$ skew-symmetric matrix $\mathbf{A}$,
    \[
        \det(\mathbf{A}) = \pf(\mathbf{A})^2 
        \ \text{ and } \
        \det(\mathbf{A}_j^i)
        =
        \begin{cases}
            (-1)^{\indicator{i>j}} \pf(\mathbf{A}_{ij}^{ij}) \pf(\mathbf{A}) & \text{$n$ is even}, \\
            \pf(\mathbf{A}_i^i) \pf(\mathbf{A}_j^j) & \text{$n$ is odd},
        \end{cases}
    \]
    where $\mathbf{A}_j^i$ is the submatrix obtained by removing the $i$-th column and $j$-th row.
    See Appendix~\ref{sec: realizable} for more discussions.
\end{rmk}

\begin{lem}\label{lem:rGP4}
    The function $\varphi$ satisfies \ref{item:rGP4}.
\end{lem}
\begin{proof}
For any $B\in \cT_n^\sigma$ and distinct $i,j\in [n]$, by \eqref{eq:GPA}, we have
    \begin{align*}
        (-1)^{m + \indicator{i\in B}} \varphi(B \setminus \skewpair{i} \cup \skewpair{j})
        &=
        \psi(B) \psi(B\symdiff\skewpairs{i}{j}) \\
        &=
        (-1)^{m + \indicator{j\in B}} \varphi(B \setminus \skewpair{i} \cup \skewpair{j}),
    \end{align*}
    where $m = \smaller{B\cap[n]}{i} + \smaller{B\cap[n]}{j} + 1$.
\end{proof}

To simplify the calculation in the proofs of Lemma~\ref{lem:rGP2} and Lemma~\ref{lem:composition1}, we introduce the following formula for $\varphi$, which unifies the two cases $B\in\cT_n^\sigma$ and $A\in\cA_n$.
\begin{lem}
Suppose $T \in \cT_n^{1-\sigma}$ and $i,j\in T$. Then  
\begin{equation}\label{eq:WtoGP}
    \varphi(T \setminus \{i\} \cup \{j^*\})
    =
        (-1)^{m} \psi(T\symdiff\skewpair{i}) \psi(T\symdiff\skewpair{j}),
\end{equation}
where $m = \smaller{T\cap[n]}{\ol{i}} + \smaller{T\cap[n]}{\ol{j}} + \indicator{i<j^*} + \indicator{i\in [n]}$. 
\end{lem}
\begin{proof}
When $i=j$, this is $\varphi(B) = \psi(B)^2$, where $B=T \setminus \{i\} \cup \{j^*\}\in\cT_n^\sigma$. 

When $i\neq j$, let $A=T \setminus \{i\} \cup \{j^*\}$ and $B=T\symdiff\skewpair{j}$. Note that $i,j^*\in B$. Then by \eqref{eq:GPA},
\[\varphi(T \setminus \{i\} \cup \{j^*\})=\varphi(A) = (-1)^{\smaller{B\cap[n]}{\ol{i}} + \smaller{B\cap[n]}{\ol{j}}+\indicator{i\in [n]}+1 }\psi(T\symdiff\skewpair{i}) \psi(T\symdiff\skewpair{j}).
\] 
We also have 
\begin{align*}
 \smaller{B\cap[n]}{\ol{i}} & \equiv  \smaller{T\cap[n]}{\ol{i}} + \smaller{(B\symdiff T)\cap[n]}{\ol{i}}\\ 
 & \equiv \smaller{T\cap[n]}{\ol{i}} + \indicator{\ol{j}<\ol{i}} \pmod{2}
\end{align*}
and 
\begin{align*}
 \smaller{B\cap[n]}{\ol{j}} & \equiv  \smaller{T\cap[n]}{\ol{j}} + \smaller{(B\symdiff T)\cap[n]}{\ol{j}}\\ 
 & \equiv \smaller{T\cap[n]}{\ol{j}} \pmod{2}
\end{align*}
Since $i\neq j$, we have $\indicator{\ol{j}<\ol{i}}=\indicator{i<j^*}+1$ and  hence
\[ \smaller{B\cap[n]}{\ol{i}} + \smaller{B\cap[n]}{\ol{j}}+\indicator{i\in [n]}+1
\equiv m \pmod{2}.
\qedhere
\]
\end{proof}

\begin{lem}\label{lem:rGP2}
    The function $\varphi$ satisfies \ref{item:rGP2}.
\end{lem}

\begin{proof}
    Let $\hyper$ be a hyper-transversal and $\hypo$ be a hypo-transversal of $\ground$. We need to show
    \begin{equation}\label{eq:WtoGP-pf}
        \sum_{k\in \hyper \setminus \hypo}
        (-1)^{\smaller{\hyper \symdiff \hypo}{k}}
        \varphi(\hyper \setminus \{k\})
        \varphi(\hypo \cup \{k\})
        \in N_F.
    \end{equation}
    Let $i,j\in\ground$ be the elements such that $\skewpair{i} \subseteq \hyper$ and $\skewpair{j} \cap \hypo = \emptyset$.
    By Lemma~\ref{lem:hypo/hyper-transversals}, we may assume that $\hyper\setminus\{i\}$ or $\hyper \setminus\{i^*\}$ is a basis because otherwise the left-hand side of \eqref{eq:WtoGP-pf} is zero. Furthermore, by interchanging $i$ and $i^*$ if necessary, we may assume that $\hyper \setminus\{i^*\}$ is a basis.
    Similarly, we may assume that $\hypo \cup\{j\}$ is a basis.
    Let \[T := \hyper \setminus \{i\} \text{ and } T' := \hypo \cup\{j^*\} .\]
    Then $T$ and $T'$ are transversals in $\cT_n^{1-\sigma}$ where $\sigma \in\{0,1\}$ in \ref{item:rGP3}.
    Since $\varphi(T)=\varphi(T') = 0$, \eqref{eq:WtoGP-pf} is equivalent to
    \begin{equation}\label{eq:WtoGP-pf1}
        \sum_{k\in T\setminus T'}
        (-1)^{\smaller{\hyper\symdiff \hypo}{k}}
        \varphi(\hyper \setminus \{k\})
        \varphi(\hypo \cup \{k\})
        \in N_F.
    \end{equation}

    By \eqref{eq:WtoGP}, for each $k \in T$,
    \[
        \varphi(\hyper \setminus \{k\})=\varphi(T \setminus\{k\}\cup \{i\})
        =
        (-1)^{m_1(k)}
        \psi(T \symdiff \skewpair{i}) \psi(T \symdiff \skewpair{k}),
    \]
    where $m_1(k) := \smaller{T \cap[n]}{\ol{k}} + \smaller{T \cap[n]}{\ol{i}} + \indicator{k<i} + \indicator{k\in [n]}$.
    Similarly, for each $k \notin T'$,
    \[
        \varphi(\hypo \cup \{k\})
        =\varphi(T'\setminus\{j^*\}\cup \{k\})=
        (-1)^{m_2(k)}
        \psi(T'\symdiff \skewpair{j}) \psi(T'\symdiff \skewpair{k}),
    \]
    where $m_2(k) := \smaller{T'\cap[n]}{\ol{j}} + \smaller{T'\cap[n]}{\ol{k}} +  \indicator{j^*<k}+ \indicator{j^*\in [n]}$.
    
    By the Wick relations~\ref{item:W2}, we have 
    \begin{equation}\label{eq:WtoGP-pf2}
        \sum_{k\in T \setminus T'} (-1)^{|(T \symdiff T')\cap[n] < \ol{k}|} \psi(T \symdiff \skewpair{k}) \psi(T'\symdiff \skewpair{k}) \in N_F.
    \end{equation}
    To derive \eqref{eq:WtoGP-pf1} from \eqref{eq:WtoGP-pf2}, it remains to show that the sign
    \[(-1)^{\smaller{\hyper \symdiff \hypo}{k}+m_1(k)+m_2(k)-\smaller{(T \symdiff T')\cap [n]}{\ol{k}}}\]
    does not depend on $k$. We simplify the exponent:
    \begin{align*}
        &\smaller{\hyper \symdiff \hypo}{k}+m_1(k)+m_2(k)-\smaller{(T \symdiff T')\cap [n]}{\ol{k}}\\
        \equiv& \smaller{\hyper}{k}+\smaller{\hypo}{k}+m_1(k)+m_2(k)-\smaller{T \cap [n]}{\ol{k}} - \smaller{T'\cap [n]}{\ol{k}}\\
        \equiv&\smaller{T}{k}+\indicator{i<k}+\smaller{T'}{k}+\indicator{j^*<k}\\
        &+\smaller{T \cap [n]}{\ol{i}}+\indicator{k<i} + \indicator{k\in [n]}+\smaller{T'\cap [n]}{\ol{j}} +  \indicator{j^*<k}+ \indicator{j^*\in [n]}\\
        \equiv & \smaller{T \symdiff T'}{k}+ \indicator{i<k}+\indicator{k<i} + \indicator{k\in [n]}+c \pmod 2,
    \end{align*}
where $c=\smaller{T \cap [n]}{\ol{i}}+\smaller{T'\cap [n]}{\ol{j}}+ \indicator{j^*\in [n]}$ does not depend on $k$. Because $T$ and $T'$ are both transversals, $T \symdiff T'$ consists of skew pairs. Then because $k\in T \setminus T'$, we have \[\smaller{T \symdiff T'}{k}\equiv \smaller{\skewpair{k}}{k}\equiv \indicator{k^*<k}\equiv 1-\indicator{k\in [n]}\pmod 2.\]
Lastly, because $k\in T=S\setminus\{i\}$, we have $k\neq i$ and hence
\[\indicator{i<k}+\indicator{k<i}=1. \qedhere \]

\end{proof}

\subsubsection{Restricted GP functions to orthogonal signatures}\label{sec:GPtoO}
Let $\varphi : \cT_n \cup \cA_n \to F$ be a restricted Grassmann--Pl\"{u}cker function. We define an $F$-signature $\cC \subseteq F^{\ground}$ as follows. For each hyper-transversal $S$, we define a vector $X_S \in F^{\ground}$ supported by $S$ as
\[
    X_S(i) := (-1)^{|S < i|} \varphi(S \setminus \{i\}) \text{ for } i\in S.
\]
Let $\cC$ be the set of all nonzero scalar multiples of such vectors $X_S$ with $X_S \ne 0$, i.e., \[ \cC = \{c X_S \mid c\in F^\times, \ S \text{ is a hyper-transversal}, \ X_S \ne 0\},\] and we define \[ \GPtoO(\eqcls{\varphi}) := \cC.\]

Let $M$ be the underlying even antisymmetric matroid of $\varphi$. By Lemma~\ref{lem:hypo/hyper-transversals}, we have the following alternative formula for $\cC$:
\[ \cC = \{c X_S \mid c\in F^\times, \ S \text{ is a hyper-transversal containing a transversal basis of }M\}.\]

\begin{lem}\label{lem:GPtoOsupport}
The supports of vectors in $\cC$ form the circuit set of the orthogonal matroid that corresponds to $M$ via the bijection in Theorem~\ref{thm:support}.
\end{lem}
\begin{proof}
By Lemma~\ref{lem:two circuits}, the support of every $X_S\in\cC$ is a fundamental circuit of $M$. Since every circuit is a fundamental circuit, the supports of vectors in $\cC$ form the circuit set of $M$. By Theorem~\ref{thm:support}, it is also the circuit set of the corresponding orthogonal matroid.    
\end{proof}

For a hypo-transversal $S'$, we define a vector $Y_{S'} \in F^{\ground}$ supported by the complement of $(S')^*$ as
\[
    Y_{S'}(i) := (-1)^{\smaller{S'}{i^*}} \varphi(S' \cup \{i^*\}) 
    \text{ for } 
    i \notin (S')^*.
\]

\begin{lem}\label{lem: vector from hypo-transversal}
    Let $S$ be a hyper-transversal and $S'$ be a hypo-transversal such that $S' = S \setminus \skewpair{i}$ for some $i$.
    Then $Y_{S'}$ equals $X_S$ or $-X_S$.
\end{lem}
\begin{proof}
    Let $\sigma\in\{0,1\}$ be as in \ref{item:rGP3}.
    By interchanging $i$ and $i^*$ if necessary, we may assume that $S\setminus\{i^*\} \in \cT_n^\sigma$.
    Note that $S$ and $(S')^*$ partition $\ground$.
    
    By definition, it suffices to show that 
    \[
        (-1)^{\smaller{S}{j}} \varphi(S \setminus \{j\})
        =
        (-1)^{\indicator{i\in[n]} + \smaller{S'}{j^*}} \varphi(S' \cup \{j^*\})
    \]
    for any $j\in S$.
    It can be shown by case analysis.
    \begin{enumerate}[label=\rm Case~\arabic*.]
        \item $j \notin \skewpair{i}$. We have 
        \[
            \varphi(\hyper \setminus \{j\}) = (-1)^{\indicator{i\in[n]} + \indicator{j\in[n]}} \varphi(\hypo \cup \{j^*\})
        \]
        by~\ref{item:rGP4} applied to $B=S \setminus \{i^*\}$, and $\smaller{\hyper}{j} + \smaller{\hypo}{j^*} \equiv \indicator{j \in[n]} \pmod{2}$.
        \item $j=i$. We have $\varphi(\hyper \setminus \{i\})=\varphi(\hypo \cup\{i^*\})=0$ by~\ref{item:rGP3}. 
        \item $j=i^*$. We have $\hyper \setminus\{i^*\} = \hypo \cup\{i\}$ and $\smaller{\hyper}{i^*} + \smaller{\hypo}{i} \equiv \indicator{i\in[n]} \pmod{2}$. \qedhere
    \end{enumerate}
\end{proof}

\begin{prop}\label{prop:GPtoO}
    $\GPtoO(\eqcls{\varphi})$ is an orthogonal $F$-signature of an orthogonal matroid.
\end{prop}

\begin{proof}
    By Lemma~\ref{lem:GPtoOsupport}, it suffices to show that $\cC$ is orthogonal, i.e.,
    \[\langle X, Y \rangle = \sum_{z\in\ground} X(z)Y(z^*) \in N_F\] for all $X,Y\in \cC$.
    By definition and Lemma~\ref{lem: vector from hypo-transversal}, we may assume that $X:=X_{\hyper}$ and $Y:=Y_{\hypo}\in \cC$ for some hyper-transversal $\hyper$ and hypo-transversal $\hypo$.

    Note that $X(z)\neq 0$ implies $z\in \hyper$, and $Y(z^*)\neq 0$ implies $z \notin \hypo$. Thus,
    \begin{equation*}
    \begin{split}
    \sum_{z\in\ground} X(z)Y(z^*) & = \sum_{z\in \hyper \setminus \hypo} X(z)Y(z^*) \\
    & = \sum_{z\in \hyper \setminus \hypo} (-1)^{\smaller{\hyper}{z}}\varphi(\hyper \setminus \{z\})\cdot (-1)^{\smaller{\hypo}{z}}\varphi(\hypo \setminus \{z\})
    \in N_F,
    \end{split}
    \end{equation*}
    where the last containment follows from~\ref{item:rGP2}.
\end{proof}

\subsubsection{Composing the maps}\label{sec:2identitymaps}
We have defined the maps $\WtoGP$, $\GPtoO$, and $\WtoO$. We denote by $\OtoW$ the inverse map of $\WtoO$ (cf. Theorem~\ref{thm:JK}). 

\begin{lem}\label{lem:composition1}
    $\GPtoO \circ \WtoGP \circ \OtoW$ is the identity map.
\end{lem}
\begin{proof}
    We will show that $\GPtoO \circ \WtoGP = \WtoO$.
    Denote $\psi$ be a Wick $F$-function. Let 
    \begin{align*}
        \eqcls{\varphi} := \WtoGP(\eqcls{\psi}), \quad
        \cC := \GPtoO(\eqcls{\varphi}), \quad
        \text{and} \quad
        \cC' := \WtoO(\eqcls{\psi}).
    \end{align*}
    By Theorem~\ref{thm:support}, the underlying orthogonal matroids of $\cC$ and $\cC'$ are the same. It suffices to check that if vectors $X\in \cC$ and $X'\in \cC'$ have the same support, then $X=cX'$ for some $c\in F^\times$.

    Take $B \in \supp(\psi)$ and $j\in B$.
    Let $T := B\symdiff\skewpair{j}$ and $S := B \cup \{j^*\}$.
    Let $X_T \in \cC'$ and $X_S \in \cC$ be the vectors defined by $\WtoO$ and $\GPtoO$, respectively. By Lemma~\ref{lem:two circuits}, 
    $X_T$ and $X_S$ have the same support.
    We claim that $X_T =c X_S$ for some $c\in F^\times$.
    
    For each $i\in T \setminus \{j^*\}$, by definition and \eqref{eq:WtoGP}, we have 
    \[
        \frac{X_T(i)}{X_T(j^*)} = (-1)^{\smallereq{T^*\cap[n]}{\ol{i}} + \smallereq{T^*\cap[n]}{\ol{j}}} \frac{\psi(T\symdiff\skewpair{i})}{\psi(T\symdiff\skewpair{j})}
    \]
    and
    \begin{align*}
        (-1)^{\smaller{S}{i} + \smaller{S}{j^*}} \frac{X_S(i)}{X_S(j^*)} 
        &= 
        \frac{\varphi(S\setminus \{i\})}{\varphi(S\setminus \{j^*\})}=\frac{\varphi(T\setminus \{i\}\cup\{j\})}{\varphi(B)} \\
        &=
        (-1)^{\smaller{T\cap[n]}{\ol{i}} + \smaller{T\cap[n]}{\ol{j}} + \indicator{i<j} +  \indicator{i\in [n]}} \frac{\psi(T\symdiff\skewpair{i}) \psi(T\symdiff\skewpair{j})}{\psi(B)^2}.
    \end{align*}
    Therefore, it suffices to check that
    \begin{multline}\label{eq: sign}
        \smaller{S}{i} + \smaller{S}{j^*} \equiv  \smallereq{T^*\cap[n]}{\ol{i}} + \smallereq{T^*\cap[n]}{\ol{j}} \\
       +\smaller{T\cap[n]}{\ol{i}} + \smaller{T\cap[n]}{\ol{j}} + \indicator{i<j} +  \indicator{i\in [n]}
        \pmod{2}.
    \end{multline}
    This follows from the following observations:
    \begin{itemize}
        \item $\smaller{S}{i} = i-1 + \indicator{\ol{j} < \ol{i}}=i-1 + \indicator{j < i}$.
        \item $\smaller{S}{j^*} = j-1 + \indicator{j\in[n]}$.
        \item $\smallereq{T^*\cap[n]}{\ol{i}} +
        \smaller{T\cap[n]}{\ol{i}}=\indicator{i\in [n]^*}+\smaller{T^*\cap[n]}{\ol{i}} +
        \smaller{T\cap[n]}{\ol{i}}=\indicator{i\in [n]^*}+i-1$.
        \item $\smallereq{T^*\cap[n]}{\ol{j}} +
        \smaller{T\cap[n]}{\ol{j}}=\indicator{j\in [n]}+\smaller{T^*\cap[n]}{\ol{j}} +
        \smaller{T\cap[n]}{\ol{j}}=\indicator{j\in [n]}+j-1$.
        
    \end{itemize}

\end{proof}

To prove Lemma~\ref{lem:composition2}, we need Corollary~\ref{cor: equalratio} which follows from the connectivity of the basis graph of an antisymmetric matroid. 
\begin{defn}[\cite{Kim}]
    Let $M$ be an antisymmetric matroid with the basis set $\cB$. The \emph{basis graph} of $M$ is a graph $G_M$ where the vertices are elements in $\cB$ and two vertices $B$ and $B'$ are adjacent if and only if $|B\setminus B'|=1$ and at least one of $B$ and $B'$ is a transversal.
\end{defn}

\begin{lem}[\cite{Kim}]
    The basis graph $G_M$ is connected.
\end{lem}

\begin{cor}\label{cor: equalratio}
Let $\varphi$ and $\varphi'$ be two restricted Grassmann--Pl\"{u}cker $F$-functions supported by the same even antisymmetric matroid $M=(\ground,\cB)$. Then $[\varphi]=[\varphi']$ if 
    \begin{equation}\label{eq: equalfraction}
        \frac{\varphi(B\setminus\skewpair{i} \cup\skewpair{j})}{\varphi(B)}
        =
        \frac{\varphi'(B\setminus\skewpair{i} \cup\skewpair{j})}{\varphi'(B)}
    \end{equation}
holds for any $B\in \cB\cap\cT_n$ and any $i,j\in B$ such that $B\setminus\skewpair{i} \cup\skewpair{j}\in\cB\cap\cA_n$.
\end{cor}
\begin{proof}
Notice that when $M$ is even, the two endpoints of any edge in the basis graph $G_M$ must be $B\in \cB\cap\cT_n$ and $B\setminus\skewpair{i} \cup\skewpair{j}\in\cB\cap\cA_n$ for some $B,i,j$. Hence we have the desired result due to the above lemma.
\end{proof}

\begin{lem}\label{lem:composition2}
  $\WtoGP \circ \OtoW \circ \GPtoO$ is the identity map.
\end{lem}
\begin{proof}
    Let $\varphi$ be a restricted Grassmann--Pl\"{u}cker $F$-function. Let 
    \begin{align*}
        \cC := \GPtoO(\eqcls{\varphi}), \quad
        \eqcls{\psi} := \OtoW(\cC), \quad
        \text{and} \quad
        \eqcls{\varphi'} := \WtoGP(\eqcls{\psi}).
    \end{align*}
    The target is to show $[\varphi]=[\varphi']$. By Theorem~\ref{thm:support}, the two functions have the same support $\cB$. By Corollary~\ref{cor: equalratio}, it suffices to show that \eqref{eq: equalfraction} holds for any $B\in \cB\cap\cT_n$ and any $i,j\in B$ such that $B\setminus\skewpair{i} \cup\skewpair{j}\in\cB\cap\cA_n$. 
    Denote $S := B \cup \{j^*\}$ and $T := B\symdiff\skewpair{j}$.
    Then we have 
    \begin{align*}
        \frac{\varphi(B\setminus\skewpair{i} \cup\skewpair{j})}{\varphi(B)}
        &=
        (-1)^{\smaller{S}{i} + \smaller{S}{j^*}}
        \frac{X_S(i)}{X_S(j^*)}, 
        \\
        \frac{\varphi'(B\setminus\skewpair{i} \cup\skewpair{j})}{\varphi'(B)}& =(-1)^{\smaller{T\cap[n]}{\ol{i}} + \smaller{T\cap[n]}{\ol{j}} +\indicator{i<j}+ \indicator{i\in [n]}}\frac{\psi(T)\symdiff\skewpair{i}\psi(T\symdiff\skewpair{j})}{\psi(T\symdiff\skewpair{j})^2},
        \\
        \frac{X_T(i)}{X_T(j^*)}
        &=
        (-1)^{\smallereq{T^*\cap[n]}{\ol{i}} + \smallereq{T^*\cap[n]}{\ol{j}}}
        \frac{\psi(T\symdiff\skewpair{i})}{\psi(T\symdiff\skewpair{j})}.       
    \end{align*}
    By Lemma~\ref{lem:two circuits}, the supports of $X_S$ and $X_T$ are the same, and hence by Lemma~\ref{lem:scalar}, $X_S$ and $X_T$ only differ by a scalar. By \eqref{eq: sign}, one can deduce \eqref{eq: equalfraction}.
\end{proof}

\begin{proof}[\bf Proof of Theorem~\ref{thm:main}]
By Lemma~\ref{lem:composition1} and Lemma~\ref{lem:composition2},
$\GPtoO \circ (\WtoGP \circ \OtoW$) and $(\WtoGP \circ \OtoW) \circ \GPtoO$ are identities maps. Hence $\GPtoO$ and $\WtoGP \circ \OtoW$ are inverse to each other. Since $\OtoW$ is bijective, $\GPtoO$ and $\WtoGP$ are also bijective. 
\end{proof}

\begin{rmk}
For a matroid-like object over tracts, we usually need a homotopy theorem on the basis graph to prove the cryptomorphisms. For matroids, Baker and Bowler \cite{BB} used Maurer's homotopy theorem \cite{Maurer1973}; for orthogonal matroids, Jin and the second author \cite{JK} used Wenzel's homotopy theorem \cite{Wenzel1995}; for antisymmetric matroids, the second author built a homotopy theorem \cite{Kim} to prove the desired cryptomorphism. Seemingly, our proof does not use a homotopy theorem, but we use the cryptomorphism in \cite{JK}.

\end{rmk}

\section{Weak restricted Grassmann--Pl\"ucker functions}\label{sec: weak}

The target of this section is to establish the ``weak'' counterpart of Theorem~\ref{thm:main}, which is Theorem~\ref{thm:main-weak}. As noted in \cite{BB}, there are (at least) two natural notions of matroids over tracts, which they call weak and strong $F$-matroids. Each of them admits a number of cryptomorphisms. Although strong matroids are all weak, the proofs for the cryptomorphisms of weak matroids turn out to be more involved. Similar situations occur in the study of orthogonal matroids \cite{JK} and our study. 

In \S\ref{sec: weak rGP}, we first define weak restricted Grassmann--Pl\"ucker functions, which are the new object in Theorem~\ref{thm:main-weak}. For the other two objects in the theorem, weak orthogonal $F$-matroids and weak orthogonal $F$-circuit sets, we recall them in \S\ref{sec: weak OM}. Then we prove that all three objects are equivalent in \S\ref{sec:weak main proof}. 

\subsection{Weak restricted Grassmann--Pl\"ucker functions}\label{sec: weak rGP}

\begin{defn}
    A \emph{weak restricted Grassmann--Pl\"ucker function} on $\ground$ over a tract $F$ is a function $\varphi : \cT_n \cup \cA_n \to F$ such that the support of $\varphi$ is the set of bases of an even antisymmetric matroid on $\ground$ and it satisfies \ref{item:rGP3}, \ref{item:rGP4}, and the following weaker replacement of \ref{item:rGP2}:
    \begin{enumerate}[label=\rm(rGP\arabic*$'$)]
        \setcounter{enumi}{1}
        \item\label{item:rGP2'} 
        For any hyper-transversal $\hyper$ and hypo-transversal $\hypo$ 
        with $|\hyper\setminus \hypo|\leq 4$, 
        we have 
        \[
            \sum_{i\in \hyper\setminus \hypo} (-1)^{\smaller{\hyper\symdiff \hypo}{i}} \varphi(\hyper\setminus\{i\}) \varphi(\hypo\cup\{i\}) \in N_F.
        \]
    \end{enumerate}
    Two weak restricted Grassmann--Pl\"ucker functions $\varphi$ and $\varphi'$ are \emph{projectively equivalent} if there is $c\in F^\times$ such that $\varphi' = c \varphi$.
\end{defn}

\begin{rmk}
In \ref{item:rGP2'}, the assumption $|\hyper\setminus \hypo|\leq 4$ can be replaced by $|\hyper\setminus \hypo|=3$ or $4$, because $|\hyper\setminus \hypo|$ is at least $2$ and the case $|\hyper\setminus \hypo|=2$ holds trivially.
\end{rmk}

Since the support of a (strong) restricted Grassmann--Pl\"ucker function forms an even antisymmetric matroid, we deduce that:

\begin{prop}
    Every strong restricted Grassmann--Pl\"ucker function is weak.
\end{prop}

\subsection{Weak orthogonal \texorpdfstring{$F$}{F}-matroids and weak orthogonal \texorpdfstring{$F$}{F}-circuit sets}\label{sec: weak OM}

We recall the cryptomorphism between weak orthogonal $F$-matroids and weak orthogonal $F$-circuit sets, which is established in \cite{JK}.  

\begin{defn}
    A \emph{weak Wick function} on $\ground$ over a tract $F$ is a function $\psi : \cT_n \to F$ such that the support of $\psi$ is the set of bases of an orthogonal matroid on $\ground$ and 
    \begin{enumerate}[label=\rm(W\arabic*$'$)]
        \setcounter{enumi}{1}
        \item\label{item:W2'} $\psi$ satisfies the \emph{$4$-term Wick relations}; that is, for any $T,T'\in \cT_n$ with $|(T\symdiff T') \cap [n]| = 4$, we have 
        \[
            \sum_{i\in(T\symdiff T')\cap [n]} (-1)^{|(T\symdiff T')\cap [n]<i|} \psi(T\symdiff\skewpair{i}) \psi(T'\symdiff\skewpair{i}).
        \]
    \end{enumerate}
\end{defn}

\begin{defn}
A \emph{weak orthogonal $F$-matroid} is an equivalence class of weak Wick $F$-functions, where two weak Wick functions are equivalent if one is a nonzero multiple of the other.    
\end{defn}

As the support of a (strong) orthogonal $F$-matroid is an orthogonal matroid, every orthogonal $F$-matroid is a weak orthogonal $F$-matroid.
Weak and strong orthogonal $F$-matroids coincide when $F$ is a partial field~\cite{BJ2023} or the tropical hyperfield $\bT$~\cite{Rincon2012}, but they differ in general~{\cite[Ex.~4.22]{JK}}. It remains open whether weak and strong orthogonal $\bS$-matroids are the same.

A weaker notion of orthogonal $F$-signatures of orthogonal matroids is defined nontrivially by making use of fundamental circuits with respect to a common basis.
Let $M = (\ground, \cB)$ be an orthogonal matroid.
We denote by $\FC(B,e)$ the fundamental basis with respect to a basis $B$ and an element $e\notin B$ (see Def.~\ref{def: fundamental circuit}).
Given circuits $C_1,\dots,C_k$ are \emph{modular} if they are distinct fundamental circuits with respect to the same basis, i.e., there are $B\in \cB$ and distinct $e_1,\dots,e_k\in B^*$ such that $C_i = \FC(B,e_i)$ for each $i$.
A tuple of modular circuits is of the \emph{first kind} if pairwise unions of the circuits are subtransversals. It is of the \emph{second kind} if none of pairwise unions of the circuits is a subtransversal. 
Remark that a modular tuple can be of neither the first kind nor the second kind, but a modular pair is always of the first or second kind by definition.

We write $\ul{X}$ for the support of a vector $X\in F^{\ground}$.
For aesthetic reason, we use $\ul{X}_i$ rather than~$\ul{X_i}$.

Later in the proofs, we will frequently use the following observation. 

\begin{lem}\label{lem: modular pair of the second kind}
For a modular pair of circuits $C_1=\FC(B,e_1)$ and $C_2=\FC(B,e_2)$, the following are equivalent:
\begin{enumerate}
    \item The modular pair is of the second kind.
    \item $|C_1\cap C_2^*|=2$.
    \item $B\symdiff\skewpairs{e_1}{e_2}$ is also a basis. 
\end{enumerate}
\end{lem}

\begin{defn}[{\cite[Def.~3.10 and Lem.~4.13]{JK}}]
    A \emph{weak $F$-circuit set} of an orthogonal matroid $M$ is an $F$-signature $\cC$ of $M$ satisfying the following conditions:
    \begin{enumerate}[label=\rm(O${}_2'$)]
       \item\label{item:O2'}
       $\langle X, Y \rangle \in N_F$ for all $X,Y\in \cC$ such that $\ul{X}$ and $\ul{Y}$ are a modular pair (of the second kind).
    \end{enumerate}
    \begin{enumerate}[label=\rm(L-\roman*$'$)]
        \item\label{item:L1}
        Let $X_1,X_2\in \cC$ be $F$-circuits such that $\underline{X}_1,\underline{X}_2$ form a modular pair of the first kind and let $f\in \underline{X}_1 \cap \underline{X}_2$.
        If $X_1(f) + X_2(f) \in N_F$, then there is $X_3 \in \cC$ such that $X_3(f)=0$ and $X_1+X_2+X_3 \in (N_F)^E$.
        \item\label{item:L2}
        Let $X_1,X_2,X_3\in \cC$ be $F$-circuits such that $\underline{X}_1,\underline{X}_2,\underline{X}_3$ form a modular triple of the second kind. Denote $\underline{X}_i = \FC(B,e_i)$ for a basis $B$ and elements $e_i\in B^*$.
        If $X_{i+1}(e_i^*) + X_{i+2}(e_i^*) \in N_F$ for each $i$, where the indices are taken modulo $3$, then there is $X_4 \in \cC$ such that $X_4(e_i^*)=0$ for each $i=1,2,3$ and $X_1+X_2+X_3+X_4 \in (N_F)^E$.
    \end{enumerate}
\end{defn}

\begin{rmk}
    Regarding \ref{item:O2'}, when $\ul{X}$ and $\ul{Y}$ are a modular pair of the first kind, we have  $\ul{X}\cap\ul{Y}^*=\emptyset$ and hence $\langle X, Y \rangle \in N_F$ holds trivially.  
\end{rmk}

\begin{rmk}
    In \ref{item:L2}, the sufficient condition can be weakened as follows: $X_{i+1}(e_i^*) + X_{i+2}(e_i^*) \in N_F$ for $i=1,2$.
    The case when $i=3$ follows from \ref{item:O2'}.
\end{rmk}

\begin{rmk}
    In~\cite{JK}, a weak $F$-circuit set of an orthogonal matroid is originally defined by replacing \ref{item:O2'} with the stronger condition:
    \begin{enumerate}[label=\rm(O${}_2$)]
        \item\label{item:O2} For all $X,Y\in \cC$, when $\langle X, Y \rangle$ has exactly two nonzero summands, $\langle X, Y \rangle\in N_F$.
    \end{enumerate}
    However, it was shown in~\cite[Lem.~4.13]{JK} that \ref{item:O2'} is equivalent to \ref{item:O2} under assuming \ref{item:L1}.
\end{rmk}

\begin{rmk}\label{rmk: weak ortho}
    In~\cite{JK}, a \emph{weak orthogonal $F$-signature} of an orthogonal matroid $M$ is defined as an $F$-signature of $M$ satisfying the orthogonality only involving at most four terms:
    \begin{enumerate}[label=\rm(O${}_4$)]
        \item For all $X,Y\in \cC$, when $\langle X, Y \rangle$ has at most four nonzero summands, $\langle X, Y \rangle\in N_F$.
    \end{enumerate}
    It was shown that weak orthogonal $F$-signatures are equivalent to \emph{moderately weak orthogonal $F$-matroids}. This notion strictly lies between strong and weak orthogonal $F$-matroids and is defined by Wick relations involving at most four nonvanishing terms determined by the underlying orthogonal matroid.
    In Definition~\ref{def: weak ortho}, we introduce a more elaborate weaker notion of orthogonal $F$-signature under the same name, which turns out to be equivalent to weak $F$-circuit sets of orthogonal matroids. We suggest calling the definition given in~\cite{JK} ``moderately weak orthogonal $F$-signatures''.

    In addition, the authors in~\cite{JK} defined a ``strong'' $F$-circuit set of an orthogonal matroid in terms of sums of vectors with modular supports, and showed that it is the same thing as an orthogonal $F$-signature.
\end{rmk}

In \S\ref{sec:review-orthogonal-matroids}, we reviewed the bijection $\WtoO$ between orthogonal $F$-matroids and orthogonal $F$-signatures (Theorem~\ref{thm:JK}).
The authors in~\cite{JK} also established a bijection between weak orthogonal $F$-matroids and weak $F$-circuit sets of orthogonal matroids, which is defined by the same formula as $\WtoO$ but with weak Wick functions instead of strong Wick functions.
We denote it by $\WtoO'$.

\begin{thm}[\cite{JK}]\label{thm:JK2}
    $\WtoO'$ is a bijection between weak orthogonal $F$-matroids and weak $F$-circuit sets of orthogonal matroids.
\end{thm}

\subsection{Proof of Theorem~\ref{thm:main-weak}}\label{sec:weak main proof}

Let $\WtoGP'$ and $\GPtoO'$ be the maps defined by the same formulas as $\WtoGP$ and $\GPtoO$ but with weak Wick functions and weak restricted Grassmann--Pl\"{u}cker functions instead of the strong ones, respectively.
We prove results analogous to Propositions~\ref{prop:WtoGP} and~\ref{prop:GPtoO}.
The proofs are basically the same but involve more technical details.

Rather than proving directly that $\GPtoO'([\varphi])$ is a weak $F$-circuit set (Prop.~\ref{prop:GPtoO2}), we introduce an intermediate step, as shown in the diagram below. 
\[\begin{tikzcd}
	{\text{Weak orth. $F$-matroids }} && {} && {\text{Weak $F$-circuit set }} \\
	\\
	{\text{Weak rGP $F$-functions }} && {} && {\text{Weak orth. $F$-signatures}}
    \arrow["\WtoO'\text{ in Thm. }\ref{thm:JK2}"', from=1-1, to=1-5]
	\arrow["\WtoGP'\text{ in Prop. }\ref{prop:WtoGP2}"', from=1-1, to=3-1]
	\arrow["\OtoW':=(\WtoO')^{-1}"', curve={height=12pt}, from=1-5, to=1-1]
	\arrow["\GPtoO'\text{ in Prop. }\ref{prop: weak rGP to weak ortho}"', from=3-1, to=3-5]
    \arrow["\begin{array}{c} \subseteq\text{ by Prop. }\ref{prop: weak ortho to weak circuit} \\ = \text{ by Cor. }\ref{cor: weak ortho equals weak circuit} \end{array}"', from=3-5, to=1-5]
\end{tikzcd}\]

We define \emph{weak orthogonal $F$-signatures} of orthogonal matroids (Def.~\ref{def: weak ortho}), which are clearly a weakening of orthogonal $F$-signatures. We then show that every weak orthogonal $F$-signature is a weak $F$-circuit set; \textit{a posteriori}, these two notions coincide  (Cor.~\ref{cor: weak ortho equals weak circuit}).

\subsubsection{Weak Wick functions to weak restricted GP functions}

\begin{prop}\label{prop:WtoGP2}
    $\WtoGP'(\eqcls{\psi})$ is an equivalence class of weak restricted Grassmann--Pl\"{u}cker function for any weak Wick $F$-function $\psi$.
\end{prop}

\begin{proof}
    Let $\eqcls{\varphi} = \WtoGP'(\eqcls{\psi})$.
    By Theorem~\ref{thm:support}, the support of $\varphi$ is the set of bases of an even antisymmetric matroid associated with the underlying orthogonal matroid of $\psi$.
    Thus, \ref{item:rGP3} holds.
    One can check \ref{item:rGP4} by the same argument as in the proof of Lemma~\ref{lem:rGP4}.

    It remains to show \ref{item:rGP2'}.
    Let $\hyper$ be a hyper-transversal and $\hypo$ be a hypo-transversal 
    such that $|\hyper \setminus \hypo| = 3$ or $4$.
    There are elements $i,j\in \ground$ such that $\skewpair{i}\subseteq \hyper$ and $\skewpair{j}\cap \hypo = \emptyset$.
    We may assume that $\hyper \setminus \{i^*\}$ and $\hypo \cup\{j\}$ are transversal bases since otherwise every term in the restricted Grassmann--Pl\"ucker relation regarding $\hyper$ and $\hypo$ vanishes by Theorem~\ref{thm:support}.
    Let $T := \hyper \setminus \{i\}$ and $T' := \hypo \cup \{j^*\}$, which are transversals.
    In the proof of Lemma~\ref{lem:rGP2}, we showed that the restricted Grassmann--Pl\"ucker relation for $\varphi$ regarding $\hyper$ and $\hypo$ is implied by the Wick relation for $\psi$ regarding $T$ and $T'$.

The same proof shows that the weak restricted Grassmann--Pl\"ucker relation \ref{item:rGP2'} is implied by the weak Wick relation \ref{item:W2'}. To be precise, 
\begin{itemize}
    \item when $|\hyper \setminus \hypo| = 3$,  
    it is easy to see that $|(T\symdiff T') \cap [n]| = |T\setminus T'| = 2$; 
    \item when $|\hyper \setminus \hypo|=4$ and $\skewpair{i} = \skewpair{j}$, 
    we also have $|(T\symdiff T') \cap [n]| = 2$; 
    \item when  $|\hyper \setminus \hypo|=4$ and $\skewpair{i} \ne \skewpair{j}$, 
    we have $|(T\symdiff T') \cap [n]| = 2$ or $4$. 
\end{itemize}
The cases $|(T\symdiff T') \cap [n]| = 2$ correspond to $2$-term Wick relations, which hold trivially. The case $|(T\symdiff T') \cap [n]| = 4$ corresponds to a $4$-term Wick relation, which holds due to \ref{item:W2'}.
\end{proof}

\subsubsection{Weak restricted GP functions to weak circuit sets}
\label{sec:GPtoO2}

\begin{prop}\label{prop:GPtoO2}
    $\GPtoO'(\eqcls{\varphi})$ is a weak orthogonal $F$-circuit set of an orthogonal matroid for any weak restricted Grassmann--Pl\"{u}cker $F$-function $\varphi$.
\end{prop}

The proof consists of two steps. First, we introduce weak orthogonal $F$-signatures (Def.~\ref{def: weak ortho}) and show that they are weak $F$-circuit sets (Prop.~\ref{prop: weak ortho to weak circuit}).
Then, it remains to prove that $\GPtoO'(\eqcls{\varphi})$ is a weak orthogonal $F$-signature (Prop.~\ref{prop: weak rGP to weak ortho}).

In order to define weak orthogonal $F$-signatures, we introduce a notion generalizing modular pairs of circuits.

\begin{defn}
    For an orthogonal matroid, we say that a transversal $T$ \emph{carries} a circuit $C$ if $C \subseteq T$ and $T \symdiff \skewpair{x}$ is a basis for each $x\in C$.
    For a nonnegative integer $k$, a pair of circuits $C_1$, $C_2$ is said to be \emph{$k$-modular} if there are transversals $T_1$, $T_2$ such that 
    \begin{itemize}
        \item $T_i$ carries $C_i$ for each $i=1,2$, and 
        \item $|T_1\cap T_2^*| \le k$.
    \end{itemize}
\end{defn}

We have the following observation.

\begin{lem}
    If a transversal $T$ carries a circuit $C$, then $T\in \cT_n^{1-\sigma}$, where $\sigma\in\{0,1\}$ is the parity of $|B\cap[n]^*|$ with any basis $B$.
    \qed
\end{lem}

The above lemma implies that $|T_1\cap T_2^*|$ is even for any transversals $T_1$, $T_2$ carrying circuits.

In particular, we use the following observation frequently in the proof of Proposition~\ref{prop: weak ortho to weak circuit}.

\begin{lem}\label{lem: vector carried by T}
    Let $\cC$ be an $F$-signature of an orthogonal matroid.
    If $T = B\symdiff \skewpair{e}$ for some basis $B$ and element $e\in B^*$, there is a vector $X\in \cC$ such that $\ul{X}$ is carried by $T$ and $X(e) \ne 0$.
\end{lem}
\begin{proof}
    Take a vector $X\in\cC$ with $\ul{X} = \FC(B,e)$.
\end{proof}

By Lemma~\ref{lem:two circuits}, two circuits are $0$-modular if and only if they are equal. 
The following lemma explains why $k$-modularity is a generalization of modular pairs.

\begin{lem}\label{lem: modular vs 2-modular}
    Two distinct circuits $C_1$ and $C_2$ are $2$-modular if and only if they are modular. 
    Moreover, $C_1$ and $C_2$ are $2$-modular and $C_1\cap C_2^* \ne \emptyset$ if and only if they are modular of the second kind.
\end{lem}
\begin{proof}
    Suppose that $C_1$ and $C_2$ are modular. Then there are a basis $B$ and distinct elements $e_1,e_2\in B^*$ such that $C_i = \FC(B,e_i)$ with $i=1,2$.
    Then the transversal $T_i := B\symdiff \skewpair{e_i}$ carries $C_i$ for each $i$, and $|T_1\cap T_2^*| = 2$. Thus, the two circuits are $2$-modular.

    Suppose that $C_1$ and $C_2$ are $2$-modular. Then there are transversals $T_1$, $T_2$ such that each $T_i$ carries $C_i$ and $|T_1\cap T_2^*| \le 2$. 
    Because $C_1 \ne C_2$, we have $|T_1\cap T_2^*| =2$.
    Note that $C_1 \cap C_2^* = \emptyset$ or $C_1 \cap C_2^* = T_1 \cap T_2^*$, since $|C_1 \cap C_2^*| \ne 1$ (the orthogonality of circuits).

    First, we consider the case $C_1 \cap C_2^* = T_1 \cap T_2^*$.
    Denote $T_1\cap T_2^* = \{e_1,e_2^*\}$.
    Then $B := T_1 \symdiff \skewpair{e_1} = T_2 \symdiff \skewpair{e_2}$ is a basis and $C_i = \FC(B,e_i)$ for each $i$.
    Thus, the circuits are modular of the second kind.

    Second, we consider the case $C_1 \cap C_2^* = \emptyset$.
    Denote $T_1\cap T_2^* = \{x,y\}$. Suppose that $C_1$ and $C_2$ are subsets of $T_1 \setminus \{x,y\} = T_2 \setminus \{x^*,y^*\}$.
    Because $C_1\ne C_2$, there are elements $a\in C_1 \setminus C_2$ and $b\in C_2 \setminus C_1$.
    Then $B_1 := T_1 \symdiff \skewpair{a}$ and $B_2 := T_2 \symdiff \skewpair{b}$ are bases, and $B_1 \setminus B_2 = \{a^*, b, x, y\}$.
    By the exchange axiom for bases, there is an element $e \in \{a^*,b,x\}$ such that both $B_1 \symdiff \skewpairs{e}{y}$ and $B_2 \symdiff \skewpairs{e}{y}$ are bases.
    However, one can check that:
    \begin{itemize}
        \item $B_1\symdiff\skewpairs{a}{y} = T_1\symdiff \skewpair{y}$ is not a basis since $y\notin C_1$;
        \item $B_2\symdiff\skewpairs{b}{y} = T_2\symdiff \skewpair{y}$ is not a basis since $y\notin C_2$; and 
        \item $B_1\symdiff\skewpairs{x}{y} = T_2\symdiff \skewpair{a}$ is not a basis since $a\notin C_2$.
    \end{itemize}
    Therefore, $C_1$ or $C_2$ cannot be a subset of $T_1 \setminus \{x,y\}$.
    By symmetry, we may assume that $x\in C_1$.
    Then $T_1 \symdiff \skewpair{x} =: B$ is a basis, and hence $C_1 = \FC(B,x)$.
    Because $T_2 \symdiff \skewpair{y} = B$, we have $C_2 = \FC(B,y^*)$.
    Therefore, $C_1$ and $C_2$ are modular of the first kind.
\end{proof}

\begin{defn}\label{def: weak ortho}
    A \emph{weak orthogonal $F$-signature}\footnote{As mentioned in Remark~\ref{rmk: weak ortho}, it is weaker than the weak orthogonal $F$-signatures defined in~\cite{JK}.} of an orthogonal matroid $M$ over a tract $F$ is an $F$-signature $\cC$ of $M$ over $F$ satisfying:
    \begin{enumerate}[label=(O${}_4'$)]
        \item\label{item:O4'} $\langle X, Y \rangle \in N_F$ for all $X,Y\in \cC$ such that $\ul{X}$ and $\ul{Y}$ are a $4$-modular pair.
    \end{enumerate}
\end{defn}

\begin{prop}\label{prop: weak ortho to weak circuit}
    Every weak orthogonal $F$-signature is a weak $F$-circuit set.
\end{prop}

\begin{proof}
    \ref{item:O2'} is a special case of \ref{item:O4'} due to Lemma~\ref{lem: modular vs 2-modular}.

    \vspace{0.2cm}

    \ref{item:L1}:
    Let $X_1,X_2\in \cC$ such that $\underline{X}_1$ and $\underline{X}_2$ are a modular pair of the first kind.
    Then, there are a transversal basis $B$ and distinct elements $e_1,e_2\in B^*$ such that such that $e_2^*\notin \underline{X}_1 = \FC(B,e_1)$ and $e_1^*\notin \underline{X}_2 = \FC(B,e_2)$.
    Note that each $\underline{X}_i$ is carried by $B \symdiff \skewpair{e_i}$.
    Let $f \in \underline{X}_1 \cap \underline{X}_2$, and suppose that $X_1(f) + X_2(f) \in N_F$.

    Let $T_3 := B \symdiff \{e_1,e_1^*,e_2,e_2^*,f,f^*\}$ and let $X_3 \in \cC$ be a vector such that $\ul{X}_3$ is carried by $T_3$ and $X_3(e_1) = -X_1(e_1)$.
    Because $T_3\symdiff\skewpair{e_1}$ is a basis, such $X_3$ exists by Lemma~\ref{lem: vector carried by T}.

    We claim $X_1 + X_2 + X_3 \in (N_F)^{\ground}$.

    When $g \in (\ground) \setminus T_3\setminus\{f\}$, we have $X_1(g) = X_2(g) = X_3(g) = 0$. 

    When $g=f^*$, we have $X_1(f^*) = X_2(f^*)=0$. We still have $X_3(f^*) = 0$ because $B\symdiff\skewpairs{e_1}{e_2}$ is not a basis. 

    When $g=f$, we have  $X_1(f)+X_2(f)\in N_F$ by the given assumption and $X_3(f) = 0$ as $f \notin T_3$. 

    When $g=e_1$, we have $X_2(e_1) = 0$ and $X_1(e_1)+X_3(e_1)\in N_F$ by the given assumption. 

    When $g=e_2$, we have $X_1(e_2)=0$. Then we show that $X_3(e_2) = - X_2(e_2)$.
    Let $Z \in \cC$ be a vector with $\ul{Z}=\FC(B,f^*)$.
    Then 
    \begin{align*}
        X_3(e_2) Z(e_2^*) &= -X_3(e_1) Z(e_1^*) 
        = X_1(e_1) Z(e_1^*) 
        = -X_1(f) Z(f^*) 
        = X_2(f) Z(f^*) 
        = -X_2(e_2) Z(e_2^*),
    \end{align*}
    where the first, third, and fifth equalities hold because $\langle X_3, Z \rangle$, $\langle X_1, Z \rangle$, and $\langle X_2, Z \rangle$ are in $N_F$ by \ref{item:O2'}, respectively.
    As $Z(e_2^*) \ne 0$, we have $X_3(e_2) = - X_2(e_2)$.
 
    When $g \in T_3 \setminus \{e_1,e_2,f^*\}$, let $T_4 := B \symdiff \skewpair{g}$ and let $Y \in \cC$ be such that $\ul{Y}$ is carried by $T_4$. 
    Note that $T_3\cap T_4^* = \{e_1,e_2,f^*,g\}$, and $X_3(f^*) = 0$ because $T_3\symdiff\skewpair{f}$ is not a basis.
    By \ref{item:O4'}, we have
    \[
        X_3(e_1) Y(e_1^*) + X_3(e_2) Y(e_2^*) + X_3(g) Y(g^*) 
        =
        \langle X_3, Y \rangle
        \in N_F,
    \]
    For each $i=1,2$, we have $X_i(e_i) Y(e_i^*) + X_i(g) Y(g^*) = \langle X_i,Y \rangle \in N_F$ by \ref{item:O2'}.
    Then $X_3(e_i) Y(e_i^*) = -X_i(e_i) Y(e_i^*) = X_i(g) Y(g^*)$.
    Since $Y(g^*) \ne 0$, we deduce that $X_1(g) + X_2(g) + X_3(g) \in N_F$.

    \vspace{0.2cm}

    \ref{item:L2}: 
    Let $X_1,X_2,X_3\in \cC$ such that $\underline{X}_1,\underline{X}_2,\underline{X}_3$ are a modular triple of the second kind.
    Then there are a transversal basis $B$ and distinct elements $e_1,e_2,e_3\in B^*$ such that $e_{i+1}^*, e_{i+2}^* \in \underline{X}_i = \FC(B,e_i)$ for each $i=1,2,3$, where the subscripts are read modulo $3$.
    Note that each $\underline{X}_i$ is carried by $B \symdiff \skewpair{e_i}$.
    Suppose that $X_{i+1}(e_i^*) + X_{i+2}(e_i^*) \in N_F$ for each $i$.
    
    Let $T_4 := B \symdiff \{e_1,e_1^*,e_2,e_2^*,e_3,e_3^*\}$ and $X_4$ be a vector such that $T_4$ carries $\ul{X}_4$ and $X_4(e_1) = -X_1(e_1)$. Because $T_4 \symdiff \skewpair{e_1}$ is a basis, such $X_4$ exists by Lemma~\ref{lem: vector carried by T}. We claim that \[X_4(e_i) = -X_i(e_i) \text{ for } i=1,2,3.\]
    The case $i=1$ is given. We will only show the case $i=2$ as the other case is similar.    
    Note that $X_4$ and $X_3$ are a modular pair by Lemma~\ref{lem: modular vs 2-modular}. Then
    \begin{align*}
        X_2(e_2) X_1(e_2^*)
        =
        - X_2(e_1^*) X_1(e_1)
        =
        - X_3(e_1^*) X_4(e_1)
        =
        X_3(e_2^*) X_4(e_2)
        =
        - X_1(e_2^*) X_4(e_2),
    \end{align*}
    where the first and third equality hold because $\langle X_2, X_1 \rangle \in N_F$ and $\langle X_3, X_4 \rangle \in N_F$ by \ref{item:O2'}, respectively.
    Since $X_1(e_2^*) \ne 0$, we have $X_2(e_2) = -X_4(e_2)$.

    Now we are ready to prove $X_1 + X_2 + X_3 + X_4 \in (N_F)^{\ground}$.
    
    When  $g \in (\ground) \setminus (T_4\cup\{e_1^*,e_2^*,e_3^*\})$, we have $X_1(g) = X_2(g) = X_3(g) = X_4(g) = 0$.

    When $g=e_i^*$ for some $i\in\{1,2,3\}$, we have $X_{1}(g) + X_{2}(g) + X_{3}(g) + X_{4}(g) = X_{i+1}(e_i^*) + X_{i+2}(e_i^*) \in N_F$ by the given assumption, where the subscripts $i+1$ and $i+2$ are read modulo $3$. 

    When $g=e_i$ for some $i\in\{1,2,3\}$, $X_{1}(g) + X_{2}(g) + X_{3}(g) + X_{4}(g) = X_{i}(e_i) + X_{4}(e_i) \in N_F$.

    When $g\in T_4 \setminus \{e_1,e_2,e_3\}$, let $Y$ be a vector such that $\ul{Y}$ is carried by $T_5:= B\symdiff \skewpair{g}$. Note that $T_4 \cap T_5^* = \{e_1,e_2,e_3,g\}$. By \ref{item:O4'}, we have 
    \[
        X_4(e_1) Y(e_1^*) + X_4(e_2) Y(e_2^*) + X_4(e_3) Y(e_3^*) +X_4(g) Y(g^*) 
        =
        \langle X_4, Y \rangle \in N_F. 
    \] 
    For each $i=1,2,3$, we have $X_i(e_i)Y(e_i^*) + X_i(g)Y(g^*) = \langle X_i, Y \rangle \in N_F$ by \ref{item:O2'}. Hence $X_4(e_i) Y(e_i^*) = - X_i(e_i) Y(e_i^*) = X_i(g) Y(g^*)$.
    Since $Y(g^*) \ne 0$, we deduce that $X_1(g) + X_2(g) + X_3(g) + X_4(g) \in N_F$.
\end{proof}

We now aim to prove the following result. 
\begin{prop}\label{prop: weak rGP to weak ortho}
    $\GPtoO'(\varphi)$ is a weak orthogonal $F$-signature of an orthogonal matroid for any weak restricted GP $F$-function $\varphi$.
\end{prop}

Recall that $\GPtoO'(\varphi)$ is an $F$-signature as defined in 
\S\ref{sec:GPtoO} except that $\varphi$ is strong there. We still adopt the notation $X_S$ and $Y_S'$ for a hyper-transversal $S$ and a hypo-transversal $S'$. 

\begin{lem}\label{lem:GPtoO2-orthogonality}
    Let $S$ be a hyper-transversal and $S'$ be a hypo-transversal. If $|S\setminus S'| \le 4$, then $\bilin{X_S}{Y_{S'}} \in N_F$.
\end{lem}
\begin{proof}
    The proof proceeds as in the second paragraph of the proof of Proposition~\ref{prop:GPtoO}.
\end{proof}

To derive Proposition~\ref{prop: weak rGP to weak ortho} from the above lemma, we have to prove that a vector $X=X_S\in\GPtoO'$ depends only on the support $\ul{X}$ rather than the choice of $S$. This is different from the strong case, where we can prove that $\cC$ is orthogonal directly and obtain this ``well-definedness'' by Lemma~\ref{lem:scalar}. The task is not simple, and we split it into the following three Lemmas. 
In these lemmas, a ``basis'' means a basis of the underlying even antisymmetric matroid of $\varphi$.

\begin{lem}\label{lem:GPtoO2-well-definedness1}
    Let $T$ be a transversal and $x_1,x_2$ be elements in $T$ such that $T\symdiff\skewpair{x_1}$ is a basis.
    Then $X_{T\cup\{x_2^*\}}  = \alpha X_{T\cup\{x_1^*\}}$ for some $\alpha\in F$.
    In particular, $\alpha \ne 0$ if $T\symdiff\skewpair{x_2}$ is a basis.
\end{lem}

\begin{proof}
    We may assume that $x_1\ne x_2$.
    Let $S_1 := T \cup \{x_1^*\}$ and $S_2 := T \cup \{x_2^*\}$.
    Denote $X_1 := X_{S_1}$ and $X_2 := X_{S_2}$ for simplicity.
    We will prove $X_{2}(y) = \alpha X_{1}(y)$ for all $y\in T \setminus \{x_1\}$, where $\alpha := \frac{X_{2}(x_1)}{X_{1}(x_1)}$. Note that $X_{1}(x_1)$ is nonzero because $S_1\setminus\{x_1\}$ is a basis. 
    
    Let $S' := S_1 \setminus \{x_1,y\}$. Note that $S_2 \setminus S' = \{x_1,x_2^*,y\}$. We denote $Y := Y_{S'}$. Then by Lemma~\ref{lem:GPtoO2-orthogonality}, we have    
    \begin{align*}
        \bilin{X_{S_1}}{Y_{S'}} 
        &= X_1(x_1) Y(x_1^*) + X_1(y) Y(y^*) \in N_F, \\
        \bilin{X_{S_2}}{Y_{S'}}
        &= X_2(x_1) Y(x_1^*) + X_2(x_2^*)Y(x_2) + X_2(y) Y(y^*) \in N_F.
    \end{align*}
    Note that $X_2(x_2^*) = 0$ because $S_2\setminus\{x_2^*\} = T$ is not a basis, and $Y(y^*) \ne 0$ because $S'\cup\{y\} = S_1\setminus\{x_1\}$ is a basis.
    Thus, from the above two relations, we deduce that
    \begin{align*}
        \frac{X_1(y)}{X_1(x_1)} = -\frac{Y(x_1^*)}{Y(y^*)}
        \quad\text{ and }\quad
        X_2(y) = -\frac{Y(x_1^*)}{Y(y^*)} X_2(x_1)
        = \frac{X_1(y)}{X_1(x_1)} X_2(x_1) = \alpha X_1(y).
    \end{align*}

    Lastly, if $T\symdiff\skewpair{x_2}$ and $T\symdiff\skewpair{x_1}$ are both bases, then $S_2\setminus\{x_1\}$ is a basis by Theorem~\ref{thm:support}. Therefore, $X_{2}(x_1)\ne 0$ and hence $\alpha = \frac{X_{2}(x_1)}{X_{1}(x_1)} \ne 0$.
    \qedhere

\end{proof}

\begin{lem}\label{lem:GPtoO2-well-definedness2}
    Let $S_1$ be a hyper-transversal and $x,x^*,y_1,\dots,y_{2k}$ be distinct elements in $S_1$ such that $S_1\setminus\{x\}$ is a transversal basis and $S_1\setminus\{y_i\}$ is not a basis for any $i=1,2,\dots,2k$.
    Let $S_2 := S_1\symdiff\{y_1,y_1^*,y_2,y_2^*,\dots,y_{2k},y_{2k}^*\}$.
    If $S_2 \setminus\{x\}$ is a basis, then $X_{S_2} = \alpha X_{S_1}$ for some $\alpha\in F^\times$.
\end{lem}
\begin{proof}
    We proceed by induction on $k\ge 1$. 
    For simplicity, we denote $X_1 := X_{S_1}$ and $X_2 := X_{S_2}$.
    
    We first check the base case, so assume $k=1$, i.e., $S_2 = S_1\symdiff \skewpairs{y_1}{y_2}$.
    We claim that $X_{2} = \alpha X_{1}$ where $\alpha := \frac{X_{2}(x)}{X_{1}(x)} \in F^\times$.

    Note that $X_1(x^*) = X_{2}(x^*) = 0$ because neither $S_1\setminus\{x^*\}$ nor $S_2 \setminus\{x^*\}$ is a basis.
    Also, $X_1(y_1) = X_1(y_2) = 0$ because neither $S_1\setminus\{y_1\}$ nor $S_1\setminus\{y_2\}$ is a basis.
    Suppose that $X_2(y_1^*) \ne 0$, i.e., $S_2\setminus\{y_1^*\}$ is a basis.
    Then $S_1\setminus \{x^*\} \symdiff \skewpair{y_2}$ is a basis by Theorem~\ref{thm:support}.
    As $S_1\setminus\{x\}$ is a basis, $S_1\setminus\{y_2\}$ is a basis by Theorem~\ref{thm:support}, a contradiction.
    Thus, $X_{2}(y_1^*) =0$ and similarly $X_{2}(y_2^*)=0$.

    If $S_1 = \{x,x^*,y_1,y_2\}$, then the supports of $X_{1}$ and $X_{2}$ are $\{x\}$, which implies that $X_{2} = \alpha X_{1}$.
    Therefore, we may assume $S_1$ is a proper superset of $\{x,x^*,y_1,y_2\}$.
    Choose an arbitrary element $z \in S_1 \setminus \{x,x^*,y_1,y_2\}$. 
    Let $S' := S_2 \setminus \{x,z\}$, which is a hypo-transversal.
    Then $S_1\setminus S' = \{x,z,y_1,y_2\}$.
    Denote $Y := Y_{S'}$.
    By Lemma~\ref{lem:GPtoO2-orthogonality}, 
    \begin{align*}
        \bilin{X_{S_2}}{Y}
        &= X_{2}(x) Y(x^*) + X_{2}(z) Y(z^*) \in N_F, \\
        \bilin{X_{S_1}}{Y}
        &= X_{1}(x) Y(x^*) + X_{1}(z) Y(z^*) + X_{1}(y_1) Y(y_1^*) + X_{1}(y_2) Y(y_2^*) \\
        &= X_{1}(x) Y(x^*) + X_{1}(z) Y(z^*) \in N_F.
    \end{align*}
    Note that $Y(z^*) \ne 0$ because $S' \cup \{z\} = S_2 \setminus \{x\}$ is a basis.
    Therefore, 
    \[
        X_i(z) = -\frac{Y(x^*)}{Y(z^*)} X_i(x)    
    \]
    for each $i=1,2$, implying that $X_{2}(z) = \alpha X_{1}(z)$.

    Next, assume that $k\ge 2$.
    Denote $B_1 :=S_1 \setminus \{x\}$ and $B_2 := S_2 \setminus \{x\}$, which are transversal bases.
    By the exchange axiom of orthogonal matroids (by identifying the given even antisymmetric matroid with an orthogonal matroid through Theorem~\ref{thm:support}), there are distinct elements $y_i,y_j\in B_1 \setminus B_2$ such that $B_3 := B_1\symdiff\skewpairs{y_i}{y_j}$ is a basis.
    Denote $S_3 := B_3 \cup \{x\}$.
    By the induction hypothesis, $X_3 := X_{S_3} = \beta X_{S_1}$ for some $\beta\in F^\times$.
    Then for each $s \in [k] \setminus \{i,j\}$, we have $X_{3}(y_s) = 0$, so $S_3 \setminus \{y_s\}$ is not a basis.
    Hence, we can apply the induction hypothesis again to deduce that  $X_{2} = \beta' X_{3}$ for some $\beta'\in F^\times$.
    Therefore, $X_{2} = \beta' \beta X_{1}$.
\end{proof}

\begin{lem}\label{lem:GPtoO2-well-definedness3}
    Let $S_1$ and $S_2$ be hyper-transversals of $\ground$.
    If the supports of ${X}_{S_1}$ and ${X}_{S_2}$ are the same, then ${X}_{S_1} = \alpha {X}_{S_2}$ for some $\alpha\in F^\times$.
\end{lem}
\begin{proof}
    Let $\skewpair{x}$ and $\skewpair{y}$ be the unique skew pairs in $S_1$ and $S_2$, respectively.
    Denote $X_1 := X_{S_1}$ and $X_2 := X_{S_2}$ for simplicity.
    We may assume that $\ul{X}_1 \ne \emptyset$.
    Hence, there is $e\in S_1$ such that $S_1\setminus\{e\}$ is a basis.
    Then either $S_1 \setminus \{x\}$ or $S_1 \setminus \{x^*\}$ is a basis by Lemma~\ref{lem:hypo/hyper-transversals}.
    By interchanging $x$ and $x^*$ if necessary, we may assume that $S_1\setminus\{x\}$ is a basis.
    Similarly, we may assume that $S_2 \setminus\{y\}$ is a basis.
    Because $\ul{X}_1 = \ul{X}_{2}$, the set $S_2\setminus\{x\}$ is a basis.
    If $\skewpair{x} \ne \skewpair{y}$, then $S_2\setminus\{y^*\} \symdiff \skewpair{x}$ is a basis by Theorem~\ref{thm:support}.
    Then by Lemma~\ref{lem:GPtoO2-well-definedness1} applied to $T = S_2\setminus\{y^*\}$, the vector $X_{2}$ equals $X_{S_2\setminus\{y^*\}\cup\{x^*\}}$ up to a nonzero scalar.
    Since $S_2\setminus\{y^*\}\cup\{x^*\}$ and $S_1$ share the same skew pair, the problem has been reduced to the case $\skewpair{x} = \skewpair{y}$. Note that we have $x=y$ in this case.

    The number $|(S_1\setminus\{x\})\setminus (S_2\setminus\{x\})|$ is even because $S_1\setminus\{x\}$ and $S_2\setminus\{x\}$ are transversal bases.
    Since $\ul{X}_{1} = \ul{X}_{2}$, we have $\varphi(S_1 \setminus\{z\}) = 0$ for each $z\in S_1\setminus S_2$, which implies that $S_1\setminus\{z\}$ is not a basis.
    Therefore, $X_{1} = \alpha X_{2}$ for some $\alpha\in F^\times$ by Lemma~\ref{lem:GPtoO2-well-definedness2}.
\end{proof}

\begin{proof}[\bf Proof of Proposition~\ref{prop: weak rGP to weak ortho}]
    Let $\cC := \GPtoO'(\varphi)$. 
    Then $\supp(\cC)$ is the circuit set of the orthogonal matroid corresponding to the underlying even antisymmetric matroid of $\varphi$ 
    via the bijection in Theorem~\ref{thm:support} (cf. Lemma~\ref{lem:GPtoOsupport}).

    Let $X_1,X_2 \in \cC$ be vectors whose supports $\ul{X}_1$ and $\ul{X}_2$ are $4$-modular.
    Then there are transversals $T_i$ with $i=1,2$ such that each $T_i$ carries $\ul{X}_i$ and $|T_1 \cap T_2^*| \le 4$.
    
    We claim that $\langle X_1, X_2 \rangle \in N_F$.
    Hence we may assume that $\ul{X}_1 \cap \ul{X}_2^* \ne \emptyset$.
    Then there are distinct elements $e_1,e_2$ in $\ul{X}_1 \cap \ul{X}_2^*$.
    Let $S := T_1 \cup \{e_1^*\}$ and $S' := T_2 \setminus \{e_2^*\}$.
    Then by Lemmas~\ref{lem:GPtoO2-well-definedness3} and~\ref{lem: vector from hypo-transversal}, $X_1 = \alpha X_{S}$ and $X_2 = \beta Y_{S'}$ for some $\alpha,\beta \in F^\times$.
    We may assume that $\alpha = \beta = 1$.
    Note that $|S \setminus S'| = |T_1 \setminus T_2| \le 4$.
    Thus, $\langle X_1, X_2 \rangle \in N_F$ by Lemma~\ref{lem:GPtoO2-orthogonality}.
\end{proof}

\begin{proof}[\bf Proof of Proposition~\ref{prop:GPtoO2}]
    It is immediate from Propositions~\ref {prop: weak ortho to weak circuit} and~\ref{prop: weak rGP to weak ortho}.
\end{proof}

\subsubsection{Composing the maps}

We have shown the well-definedness of the maps $\WtoGP'$ and $\GPtoO'$ (Prop.~\ref{prop:WtoGP2} and~\ref{prop:GPtoO2}).
Hence, Theorem~\ref{thm:main-weak} follows from the following lemma.
We denote $\OtoW' := (\WtoO')^{-1}$.

\begin{lem}
    $\GPtoO' \circ \WtoGP' \circ \OtoW'$ and $\WtoGP' \circ \OtoW' \circ \GPtoO'$ are identity maps.
\end{lem}
\begin{proof}
    The proof is the same as those of Lemmas~\ref{lem:composition1} and~\ref{lem:composition2}.
\end{proof}

\begin{proof}[\bf Proof of Theorem~\ref{thm:main-weak}]
It is a direct consequence of the above lemma. 
\end{proof}

\begin{cor}\label{cor: weak ortho equals weak circuit}
    Weak orthogonal $F$-signatures and weak $F$-circuit sets of orthogonal matroids are the same.
\end{cor}
\begin{proof}
    By Theorem~\ref{thm:main-weak}, the map $\GPtoO'$ from the set of equivalence classes of weak restricted GP $F$-functions to the set of weak $F$-circuit sets of orthogonal matroids is bijective.
    Therefore, the inclusion map shown in Proposition~\ref{prop: weak ortho to weak circuit} must be surjective.
\end{proof}

\section{Miscellaneous}\label{sec: mis}
\subsection{Connection to enveloping matroids}\label{sec:enveloping-matroids}

Given an orthogonal matroid $M = (\ground,\cB)$, an \emph{enveloping matroid}~{\cite[p.~78]{BGW2003}} of $M$ is a matroid $M' = (\ground,\cB')$ such that $\cB' \cap \cT_n = \cB$. While every representable orthogonal matroid admits an enveloping matroid, it is unknown whether this holds for all orthogonal matroids. Motivated by this, we ask a more general question.
\begin{que}
    For which tract $F$ does the following hold?
    For any orthogonal $F$-matroid $[\psi]$, there is an $F$-matroid $[\varphi]$ such that $\varphi(T) = \psi(T)^2$ for all $T\in \cT_n$.
\end{que}

Note that when $F$ is the Krasner hyperfield, the question is exactly asking whether any orthogonal matroid admits an enveloping matroid. When $F$ is a field, from the Wick function $\psi$, we may construct a skew-symmetric matrix $\mathbf{A}$ whose Pfaffians are the values of $\psi$, and the $n$-by-$2n$ matrix $(\mathbf{I}_n \mid \mathbf{A})$ represents an $F$-matroid having the desired property. 

We shall show in Corollary~\ref{cor: envelop} that we can assign values $\varphi(A)$ to $A\in \cA_n$ so that $\varphi$ satisfies the Grassmann--Pl\"ucker relations that only involve terms in $\{\varphi(X):X\in\cT_n \cup \cA_n\}$.

The following lemma shows what these relations are, which can be proved by a simple case discussion. 
\begin{lem}\label{lem:GP on TA}
    Let $\varphi: \binom{\ground}{n} \to F$ be a function. Let
        \[
            \sum_{e\in \hyper \setminus \hypo} (-1)^{|\hyper \symdiff \hypo < e|} \varphi(\hyper\setminus \{e\}) \varphi(\hypo\cup \{e\}) \in N_F.
        \]
    be a Grassmann--Pl\"{u}cker relation, where $\hyper$ and $\hypo$ are subsets of $\ground$ such that $|\hyper|=n+1$ and $|\hypo|=n-1$. If $\hyper\setminus \{e\}$ and $\hypo\cup \{e\}$ are in $\cT_n \cup \cA_n$ for every $e\in \hyper\setminus \hypo$, then $\hyper$ and $\hypo$ satisfy exactly one of the following:
    \begin{enumerate}
        \item $\hyper$ is a hyper-transversal and $\hypo$ is a hypo-transversal (as in \ref{item:rGP2}),
        
        \item $\hyper$ is a hyper-transversal containing $\skewpair{i}$, and $\skewpair{i}\subseteq \hypo\subseteq \hyper$.
       
        \item $\hyper$ is a hyper-transversal containing $\skewpair{i}$, and $\hypo=\hyper\setminus\skewpair{i}\setminus\{k\}\cup\{j^*\}$ for distinct $k,j\in \hyper\setminus\skewpair{i}$. 

        \item $\hyper$ contains exactly two skew pairs $\skewpairs{i}{j}$, and $\hypo=\hyper\setminus D$, where $D$ is a $2$-subset of $\{i,i^*,j,j^*\}$.
        
        \item $\hyper$ contains exactly two skew pairs $\skewpairs{i}{j}$, and $\hypo=\hyper\setminus \{i,i^*,j^*\}\cup \{k\}$, where $k$ satisfies $\hyper\cap\skewpair{k}=\emptyset$.
        
        \item $\hyper$ contains exactly two skew pairs $\skewpairs{i}{j}$, and  $\hypo=\hyper\setminus \skewpairs{i}{j}\cup \skewpair{k}$, where $k$ satisfies $\hyper\cap\skewpair{k}=\emptyset$. 
        \end{enumerate}
\end{lem}

\begin{prop}\label{prop:GP on TA}
Let $\varphi$ be a restricted Grassmann--Pl\"ucker function over a tract $F$ with $1+1-1-1\in N_F$. Then $\varphi$ satisfies all the Grassmann--Pl\"ucker relations on $\ground$ that only involve terms in $\{\varphi(X):X\in\cT_n \cup \cA_n\}$.
\end{prop}

\begin{proof}
We check for the six cases in Lemma~\ref{lem:GP on TA}. 
\begin{enumerate}
    \item The first case is \ref{item:rGP2}, which holds because $\varphi$ is a restricted Grassmann--Pl\"ucker function. 
    \item In the second case, $|\hyper \setminus \hypo|=2$. Hence we have a $2$-term Grassmann--Pl\"ucker relation which is of the form $\varphi(X)\varphi(Y)-\varphi(Y)\varphi(X)\in N_F$. This obviously holds.
    \item In the third case, we have $\hyper=\{i,i^*,j,k\}\cup Z$ and $\hypo=\{j,j^*\}\cup Z$, where $Z$ is a transversal of $\ground\setminus\{i,i^*,j,j^*,k,k^*\}$. For simplicity, we write $\hyper=ii^*jkZ$, $\hypo=jj^*Z$, etc. Then the Grassmann--Pl\"ucker relation is 
    \begin{align*}
    & (-1)^{\smaller{\hyper\symdiff \hypo}{i}}\varphi(i^*jkZ)\varphi(ijj^*Z)\\
    +&(-1)^{\smaller{\hyper\symdiff \hypo}{i^*}}\varphi(ijkZ)\varphi(i^*jj^*Z)\\
    +&(-1)^{\smaller{\hyper\symdiff \hypo}{k}}\varphi(ii^*jZ)\varphi(jj^*kZ)\in N_F,
    \end{align*}
    where $\hyper\symdiff \hypo=ii^*j^*k$. 
    We claim that this relation is the same as \ref{item:rGP2} for $\hyper_2:=ijj^*kZ$ and $\hypo_2:=i^*jZ$. Indeed, $\hyper_2 \symdiff \hypo_2 = \hyper \symdiff \hypo$ and the second relation is
    \begin{align*}
    & (-1)^{\smaller{\hyper_2\symdiff \hypo_2}{k}}\varphi(ijj^*Z)\varphi(i^*jkZ)\\
    +&(-1)^{\smaller{\hyper_2\symdiff \hypo_2}{j^*}}\varphi(ijkZ)\varphi(i^*jj^*Z)\\
    +&(-1)^{\smaller{\hyper_2\symdiff \hypo_2}{i}}\varphi(jj^*kZ)\varphi(ii^*jZ)\in N_F.
    \end{align*}
    It remains to check that the signs match: \begin{align*}
    (-1)^{\smaller{\hyper\symdiff \hypo}{i}}/(-1)^{\smaller{\hyper\symdiff \hypo}{k}}&=(-1)^{\smaller{\hyper_2\symdiff \hypo_2}{k}}/(-1)^{\smaller{\hyper_2\symdiff \hypo_2}{i}},\\ (-1)^{\smaller{\hyper\symdiff \hypo}{i}}/(-1)^{\smaller{\hyper\symdiff \hypo}{i^*}}&=(-1)^{\smaller{\hyper_2\symdiff \hypo_2}{k}}/(-1)^{\smaller{\hyper_2\symdiff \hypo_2}{j^*}}.
    \end{align*} 
    The first one clearly holds. The second one also holds, because $\hyper\symdiff \hypo=ii^*j^*k$ and hence \[\smaller{\hyper\symdiff \hypo}{i}+\smaller{\hyper\symdiff \hypo}{i^*}+\smaller{\hyper_2\symdiff \hypo_2}{k}+\smaller{\hyper_2\symdiff \hypo_2}{j^*}=0+1+2+3=6\]
    for any order of $i,i^*,j^*,k$. 
    \item The fourth case is again a $2$-term Grassmann--Pl\"ucker relation, which holds trivially. 
    \item In the fifth case, we have $\hyper=ii^*jj^*Z$ and $\hypo=jkZ$, where $Z$ is a transversal of $\ground\setminus\{i,i^*,j,j^*,k,k^*\}$. The only difference between the fifth case and the third case is that $j^*$ and $k$ are swapped (nothing changes for $j$ or $k^*$). To be precise, we present the initial steps. The Grassmann--Pl\"ucker relation is 
    \begin{align*}
    & (-1)^{\smaller{\hyper\symdiff \hypo}{i}}\varphi(i^*jj^*Z)\varphi(ijkZ)\\
    +&(-1)^{\smaller{\hyper\symdiff \hypo}{i^*}}\varphi(ijj^*Z)\varphi(i^*jkZ)\\
    +&(-1)^{\smaller{\hyper\symdiff \hypo}{j^*}}\varphi(ii^*jZ)\varphi(jj^*kZ)\in N_F,
    \end{align*}
    where $\hyper\symdiff \hypo=ii^*j^*k$. 
    We claim that this relation is the same as \ref{item:rGP2} for $\hyper_2:=ijj^*kZ$ and $\hypo_2:=i^*jZ$. Indeed, $\hyper_2\symdiff \hypo_2=\hyper\symdiff \hypo$ and the second relation is
    \begin{align*}
    & (-1)^{\smaller{\hyper_2\symdiff \hypo_2}{j^*}}\varphi(ijkZ)\varphi(i^*jj^*Z)\\
    +&(-1)^{\smaller{\hyper_2\symdiff \hypo_2}{k}}\varphi(ijj^*Z)\varphi(i^*jkZ)\\
    +&(-1)^{\smaller{\hyper_2\symdiff \hypo_2}{i}}\varphi(jj^*kZ)\varphi(ii^*jZ)\in N_F.
    \end{align*}
    The sign checking is similar to the third case. 
    \item In the last case, we have $\hyper=ii^*jj^*Z$ and $\hypo=kk^*Z$, where $Z$ is a transversal of $\ground\setminus\{i,i^*,j,j^*,k,k^*\}$. The Grassmann--Pl\"ucker relation is 
    \begin{align*}
    & (-1)^{\smaller{W}{i}}\varphi(i^*jj^*Z)\varphi(ikk^*Z)+(-1)^{\smaller{W}{j}}\varphi(ii^*j^*Z)\varphi(jkk^*Z)\\
    +& (-1)^{\smaller{W}{i^*}}\varphi(ijj^*Z)\varphi(i^*kk^*Z)+(-1)^{\smaller{W}{j^*}}\varphi(ii^*jZ)\varphi(j^*kk^*Z)\in N_F,
    \end{align*}
    where $W:=\hyper\symdiff \hypo=ii^*jj^*kk^*$. We shall show that this relation is of the form $X-X+X-X\in N_F$ in three steps (\eqref{eq: W1}, \eqref{eq: W4}, and \eqref{eq: W5}), which finishes the proof.  
    \begin{itemize}
        \item We first show that
        \begin{equation}\label{eq: W1}
        (-1)^{\smaller{W}{i}}\varphi(i^*jj^*Z)\varphi(ikk^*Z)=- (-1)^{\smaller{W}{j}}\varphi(ii^*j^*Z)\varphi(jkk^*Z).
        \end{equation}
        Applying \ref{item:rGP2} to $\hyper=i^*jj^*kZ$ and $T=ik^*Z$, we obtain
        \begin{align*}
        & (-1)^{\smaller{W}{k}}\varphi(i^*jj^*Z)\varphi(ikk^*Z)+(-1)^{\smaller{W}{i^*}}\varphi(jj^*kZ)\varphi(ii^*k^*Z)\\
        +& (-1)^{\smaller{W}{j}}\varphi(i^*j^*kZ)\varphi(ijk^*Z)+(-1)^{\smaller{W}{j^*}}\varphi(i^*jkZ)\varphi(ij^*k^*Z)\in N_F.
        \end{align*}
        The key observation here is that the last two terms $\varphi(i^*j^*kZ)\varphi(ijk^*Z)$ and $\varphi(i^*jkZ)\varphi(ij^*k^*Z)$ vanish
        due to \ref{item:rGP3}.
        Hence we have
        \begin{equation}\label{eq: W2}
        (-1)^{\smaller{W}{k}}\varphi(i^*jj^*Z)\varphi(ikk^*Z)=-(-1)^{\smaller{W}{i^*}}\varphi(jj^*kZ)\varphi(ii^*k^*Z).
        \end{equation}
        By swapping $i$ and $j$ (along with $i^*$ and $j^*$), we get
        \begin{equation}\label{eq: W3} (-1)^{\smaller{W}{k}}\varphi(ii^*j^*Z)\varphi(jkk^*Z)=-(-1)^{\smaller{W}{j^*}}\varphi(ii^*kZ)\varphi(jj^*k^*Z).
        \end{equation}
        Since $\sum_{e\in W}\smaller{W}{e}=0+1+2+3+4+5$ is odd,         \eqref{eq: W2} and \eqref{eq: W3} imply \eqref{eq: W1}. 
        \item By swapping $i$ with $i^*$ and $j$ with $j^*$ in \eqref{eq: W1}, we get
        \begin{equation}\label{eq: W4} (-1)^{\smaller{W}{i^*}}\varphi(ijj^*Z)\varphi(i^*kk^*Z)=-(-1)^{\smaller{W}{j^*}}\varphi(ii^*jZ)\varphi(j^*kk^*Z).
        \end{equation}
        \item Lastly, we show that 
        \begin{equation}\label{eq: W5}
        (-1)^{\smaller{W}{i}}\varphi(i^*jj^*Z)\varphi(ikk^*Z)=(-1)^{\smaller{W}{i^*}}\varphi(ijj^*Z)\varphi(i^*kk^*Z).   \end{equation}
        The order we choose on $\ground$ implies $\smaller{W}{i}-\smaller{W}{i^*}=\pm 1$. Then by Lemma~\ref{lem: rgp4}, 
        \[
        \varphi(i^*jj^*Z)=(-1)^a\varphi(i^*kk^*Z)\text{ and }
        \varphi(ikk^*Z)=(-1)^{b}\varphi(ijj^*Z),
        \]
        where $a-b=\indicator{i^*\in [n]^*}-\indicator{i\in [n]^*}=\pm 1$. Therefore, \eqref{eq: W5} holds. \qedhere
    \end{itemize}
\end{enumerate}
\end{proof}

\begin{rmk}
By the proof of the above proposition, if we only assume that $\varphi$ satisfies \ref{item:rGP2} and \ref{item:rGP3}, then we can still get the desired result for tracts with $x-x+y-y\in N_F$ for any $x,y\in F$. 
\end{rmk}

\begin{cor}\label{cor: envelop}
   Let $F$ be a tract with $1+1-1-1\in N_F$ and $\psi:\cT_n\to F$ be a Wick function. Then there exists a function $\varphi:\cT_n\cup\cA_n\to F$ such that $\varphi(T)=\psi(T)^2$ for all $T\in\cT_n$ and $\varphi$ satisfies all the Grassmann--Pl\"ucker relations on $\ground$ that only involve terms in $\{\varphi(X):X\in\cT_n \cup \cA_n\}$.
\end{cor}
\begin{proof}
Let $\varphi$ be as defined in Section~\ref{sec:WtoGP}, which implies $\varphi(T)=\psi(T)^2$ for all $T\in\cT_n$. By Proposition~\ref{prop:WtoGP}, $\varphi$ is a restricted Grassmann--Pl\"ucker function. Then by Proposition~\ref{prop:GP on TA}, we get the desired result. 
\end{proof}

Remark that the authors of~\cite{EFLS2024} studied the log-concavity of $\Delta$-matroids and orthogonal matroids by employing a slightly stronger notion of enveloping matroids.

\subsection{Almost-principal minors and Pfaffian positivity}

Boretsky et al.~\cite{BCE2024} showed that the totally positive orthogonal Grassmannian in the sense of Lusztig~\cite{Lusztig1994} can be characterized by the positivity of certain $\binom{n}{2}$ coordinates out of the $\binom{2n}{n}$ Pl\"ucker coordinates. They also found its connection with Pfaffian positivity for skew-symmetric matrices.

In the following proposition, we prove that Pfaffian positivity for a skew-symmetric matrix is equivalent to the positivity of certain almost-principal minors, which also holds for nonnegativity; cf. {\cite[Thm.~1.5]{BCE2024}}.
The proof easily follows from Cayley's Pfaffian identity (Remark~\ref{rmk: Cayley}).

For a square matrix $\mathbf{A}$, an almost-principal submatrix $\mathbf{A}[I\cup\{i\}, I\cup\{j\}]$ is called \emph{top-right} if $i<j$.

\begin{prop}
    Let $\mathbf{A}$ be an $n$-by-$n$ skew-symmetric matrix with real entries.
    Then the following are equivalent:
    \begin{enumerate}
        \item The Pfaffians of even-sized principal submatrices are positive (resp. nonnegative).
        \item The determinants of 
        odd-sized top-right almost-principal submatrices are positive (resp. nonnegative).
    \end{enumerate}
\end{prop}
\begin{proof}
    Note that the Pfaffian of the $0$-by-$0$ matrix is $1$ by convention.

    The implications of both directions for the positive case are immediate from Cayley's Pfaffian identity. 
    The implication from 1 to 2 for the nonnegative case is also straightforward from Cayley's Pfaffian identity.
    Thus, it only remains to show the converse.

    Suppose $\mathbf{A} = (a_{ij})$ is a skew-symmetric matrix whose odd-sized top-right almost-principal minors are nonnegative.
    We show that $\pf(\mathbf{A}_I) \ge 0$ for all even-sized subsets $I\subseteq [n]$. We proceed by induction on ${|I|}/{2}$. 
    The base case $I=\emptyset$ is trivial, so we assume that $|I|>0$.
    We may assume $\pf(\mathbf{A}_I) \ne 0$.
    Denote $I = \{i_1 < i_2 < \dots < i_{2k}\}$.
    Because of the recursive definition of Pfaffians, we have 
    \[
        \pf(\mathbf{A}_I) = \sum_{j=2}^{2k} (-1)^j a_{i_1 i_j} \pf(\mathbf{A}_{I\setminus\{i_1,i_j\}}).
    \]
    Then $\pf(\mathbf{A}\setminus\{i_1,i_j\})$ is nonzero for some $j$, which is indeed positive by induction hypothesis.
    By Cayley's Pfaffian identity, we conclude that 
    \[
        \pf(\mathbf{A}_{I}) = \frac{\det(\mathbf{A}[I\setminus\{i_j\}, I\setminus\{i_1\}])}{\pf(\mathbf{A}_{I\setminus\{i_1,i_j\}})} > 0.
        \qedhere
    \]
\end{proof}

\section*{Acknowledgments}
Thanks to Matt Baker for helpful discussions.

\printbibliography


\appendix

\section{\texorpdfstring{$\WtoGP$}{The function WtoGP} in the realizable case}\label{sec: realizable}

In \S\ref{sec:WtoGP}, we construct the GP function $\varphi$ from a Wick function $\psi$, and define $\WtoGP([\psi]):=[\varphi]$. The formulas involved in the construction are complicated, especially the sign in \eqref{eq:GPA}. In this appendix, we justify our construction by showing that it is correct in the realizable case.

Let $\mathbf{A}$ be an $n$-by-$n$ skew-symmetric matrix over a field $K$.
Then, the corresponding Wick $K$-function $\psi: \cT_n \to K$ is defined by
\[    \psi(([n]\setminus I) \cup I^*) := \pf(\mathbf{A}[I,I])
\]
for $I\subseteq[n]$, where $\mathbf{A}[R,C]$ denotes the submatrix with rows indexed by $R$ and columns indexed by $C$.

Let $\mathbf{A}'$ be the $n$-by-$2n$ matrix with columns indexed by $1,1^*,2,2^*,\dots,n,n^*$ in order such that the square submatrix with columns $1,\dots,n$ is the identity matrix and the square submatrix with columns $1^*,\dots,n^*$ is $\mathbf{A}$.
We consider the restricted GP $K$-function $\varphi: \cT_n \cup \cA_n \to K$ defined by
\[
    \varphi(B) := \det(\mathbf{A}'[B])
\]
for $B\in \cT_n \cup \cA_n$, where $\mathbf{A}'[B]$ denotes the square submatrix with columns indexed by $B$.
Note that $\sigma = 0$ in this case, due to $\varphi([n]) = 1 \ne 0$. 

Our target is to show that this $\varphi$ is exactly the one we defined in \S\ref{sec:WtoGP}. 

By definition, we have 
\begin{align*}
    \varphi(([n]\setminus I) \cup I^*) 
    = \det(\mathbf{A}'[([n]\setminus I) \cup I^*]) &= \det(\mathbf{A}[I,I]) \\
    &= \pf(\mathbf{A}[I,I])^2 = \psi(([n]\setminus I) \cup I^*)^2. 
\end{align*}
for $I\subseteq[n]$.
The second last equality is the first part of Cayley's Pfaffian identity (Remark~\ref{rmk: Cayley}).

It remains to show the almost-transversal case \eqref{eq:GPA}. We will use the second part of Cayley's Pfaffian identity:
\[
    \det(\mathbf{A}[Ii, Ij])
    =
    \begin{cases}
        (-1)^{\indicator{i>j}} \pf(\mathbf{A}[I,I]) \pf(\mathbf{A}[Iij, Iij]) & \text{$|I|$ is even}, \\
        \pf(\mathbf{A}[Ii,Ii]) \pf(\mathbf{A}[Ij,Ij]) & \text{$|I|$ is odd},
    \end{cases}
\]
for $I\subseteq[n]$ and distinct $i,j\in[n]\setminus I$, where we denote $I\cup\{i\}$, $I\cup\{j\}$, and $I\cup\{i,j\}$ by $Ii$, $Ij$, and $Iij$ for shorthand.
Note that 
\begin{align*}
    \varphi(([n]\setminus Ii) \cup (Ij)^*)
    &= (-1)^m \det(\mathbf{A}[Ii, Ij]),
\end{align*}
where
\[
    m= |[n]\setminus I > i| + |[n]\setminus I > j|  + \indicator{i>j}.
\]
Here, $|X>i|$ denotes the number of elements $x\in X$ with $x>i$. The number $m$ comes from the cofactor expansion along the columns indexed by $[n]\setminus Ii$.
Then, one can derive \eqref{eq:GPA} by computing the signs as follows.

Suppose that $|I|$ is even. Then,
\[
    \det(\mathbf{A}[Ii, Ij]) =
    (-1)^{\indicator{i>j}} \pf(\mathbf{A}[I,I]) \pf(\mathbf{A}[Iij, Iij]).
\]
Therefore, we have 
\begin{align}
    \varphi(([n]\setminus Ii ) \cup (Ij)^*)
    &= 
    (-1)^m \det(\mathbf{A}[Ii, Ij]) \notag \\
    &=
    (-1)^{m+ \indicator{i>j}} \pf(\mathbf{A}[I, I]) \pf(\mathbf{A}[Iij, Iij]) \notag \\
    &=
    (-1)^{m+ \indicator{i>j}} \psi(([n]\setminus I) \cup I^*) \psi(([n]\setminus Iij) \cup (Iij)^*)
    \label{eq: realizable1}
\end{align}
Let $B := ([n]\setminus I) \cup I^* \in \cT_n^0$.
Then $\indicator{i\in B} = 1$ and, for each $k\in[n]$
\begin{align*}
    \smaller{B\cap[n]}{k} 
    = |[n]\setminus I < k|
    &= |[n] < k| - |I < k| \\
    &= (n-1) - |[n] > k| + |I| - |I > k| \\
    &= (n-1) + |I| - |[n]\setminus I > k|.
\end{align*}
Hence,
\begin{align*}
    \smaller{B\cap[n]}{i} + \smaller{B\cap[n]}{j} + \indicator{i\in B} +1
    \equiv 
    |[n]\setminus I > i| + |[n]\setminus I > j| \pmod{2},
\end{align*}
showing that \eqref{eq: realizable1} coincides with \eqref{eq:GPA}.

Suppose that $|I|$ is odd. Then,
\[
    \det(\mathbf{A}[Ii, Ij]) = \pf(\mathbf{A}[Ii,Ii]) \pf(\mathbf{A}[Ij,Ij]),
\]
and hence 
\begin{align}
    \varphi(([n]\setminus Ii ) \cup (Ij)^*)
    &= 
    (-1)^m \det(\mathbf{A}[Ii, Ij]) \notag \\
    &=
    (-1)^{m} \pf(\mathbf{A}[Ii, Ii]) \pf(\mathbf{A}[Ij, Ij]) \notag \\
    &=
    (-1)^{m} \psi(([n]\setminus Ii) \cup (Ii)^*) \psi(([n]\setminus Ij) \cup (Ij)^*)
    \label{eq: realizable2}
\end{align}
Let $B := ([n]\setminus Ii) \cup (Ii)^* \in \cT_n^0$. Then $\indicator{i\in B} = 0$ and
\begin{align*}
    \smaller{B\cap[n]}{k} 
    = |[n]\setminus Ii < k|
    &= (n-1) + |Ii| - |[n]\setminus Ii > k| \\
    &= (n-1) + |Ii| - |[n]\setminus I > k| - \indicator{i>k}
\end{align*}
for $k\in[n]$. Hence, 
\begin{align*}
    \smaller{B\cap[n]}{i} + \smaller{B\cap[n]}{j} + \indicator{i\in B} +1
    &\equiv 
    |[n]\setminus I > i| + |[n]\setminus I > j| + \indicator{i>j} \pmod{2}.
\end{align*}
This shows that \eqref{eq: realizable2} coincides with \eqref{eq:GPA}.

\end{document}